\documentclass[11pt]{article}
\usepackage{amssymb}

\usepackage{amsmath}
\usepackage{theorem}
\usepackage{here}
\usepackage[pdftex]{graphicx}

\newtheorem{thm}{Theorem}[section]
\newtheorem{lem}{Lemma}[section]

\newtheorem{prp}{Proposition}[section]
\theorembodyfont{\rmfamily}
\newtheorem{algo}{Algorithm}[section]

\newtheorem{remark}{Remark}[section]
\makeatletter
\renewcommand{\theequation}{%
   \thesection.\arabic{equation}}
\@addtoreset{equation}{section}
\makeatother

\usepackage{comment} %
\usepackage{bm}

\def\ta{{\tau}}

\def\al{{\alpha}}
\def\be{{\beta}}

\def\de{{\delta}}
\def\ep{{\varepsilon}}

\def\si{{\sigma}}

\def\bbe{{\text{\boldmath $\beta$}}}

\def\bep{{\text{\boldmath $\varepsilon$}}}

\def\sih{{\hat \si}}

\def\alt{{\tilde \al}}
\def\bet{{\tilde \be}}

\def\bbeh{{\widehat \bbe}}

\def\bbet{{\widetilde \bbe}}

\def\sis{{\small \bsi}}

\def\Si{{\Sigma}}
\def\Ga{{\Gamma}}
\def\Om{{\Omega}}

\def\bSi{{\text{\boldmath $\Si$}}}
\def\bGa{{\text{\boldmath $\Ga$}}}
\def\bOm{{\text{\boldmath $\Om$}}}

\def\bPsi{{\text{\boldmath $\Psi$}}}

\def\bSit{{\widetilde \bSi}}

\def\b{{\text{\boldmath $b$}}}

\def\e{{\text{\boldmath $e$}}}

\def\i{{\text{\boldmath $i$}}}

\def\u{{\text{\boldmath $u$}}}
\def\v{{\text{\boldmath $v$}}}

\def\x{{\text{\boldmath $x$}}}
\def\y{{\text{\boldmath $y$}}}

\def\A{{\text{\boldmath $A$}}}
\def\B{{\text{\boldmath $B$}}}
\def\C{{\text{\boldmath $C$}}}

\def\I{{\text{\boldmath $I$}}}

\def\M{{\text{\boldmath $M$}}}

\def\O{{\text{\boldmath $O$}}}

\def\U{{\text{\boldmath $U$}}}
\def\V{{\text{\boldmath $V$}}}
\def\W{{\text{\boldmath $W$}}}
\def\X{{\text{\boldmath $X$}}}
\def\Y{{\text{\boldmath $Y$}}}

\def\ubh{{\hat \u}}

\def\ybt{{\tilde \y}}

\def\wt{{\tilde w}}

\def\gt{{\tilde g}}

\def\Xbt{{\widetilde \X}}

\def\Wbt{{\widetilde \W}}

\def\tr{{\rm tr\,}}

\def\etr{{\rm etr\,}}

\def\[{{\text{\boldmath $[$}}}
\def\]{{\text{\boldmath $]$}}}

\def\/{{\Bigr/\!\!}}

\def\1r{{\rm (1)}}
\def\2r{{\rm (2)}}
\def\3r{{\rm (3)}}
\def\4r{{\rm (4)}}
\def\5r{{\rm (5)}}

\def\non{{\nonumber}}
\DeclareMathOperator*{\bdiag}{{\bf Diag\,}}
\def\bBe{{\text{\boldmath $\mathcal{B}$}}}
\def\sis{{{\si ^2}}}

\begin{document}
\title{Trace-Class Results for MCMC Algorithms for Student-$t$ Regression Models}
\author{
Yasuyuki Hamura\footnote{Center for Spatial Information Science, The University of Tokyo, 
5-1-5 Kashiwanoha, Kashiwa, Chiba 277-8568, JAPAN. 
\newline{
E-Mail: yasu.stat@gmail.com}} \
}
\maketitle
\begin{abstract}
In this paper, we consider MCMC algorithms for Student-$t$ regression models. 
In three cases, we investigate the efficiency of Markov chains based on the algorithms in terms of whether trace-class results hold or not. 
First, we consider the case where the parameters follow a matrix-normal-inverse-Wishart distribution and show that the Markov operator associated with a standard data augmentation algorithm is trace-class. 
Second, we consider the case of an improper prior and univariate outcomes. 
In this case, the standard Markov operator is not trace-class but the Markov operator associated with a collapsed Gibbs algorithm is trace-class. 
Third, we consider the case of an improper prior and multivariate outcomes. 
We obtain a trace-class result for a parameter expanded data augmentation algorithm which is based on a univariate working parameter. 
Finally, we consider the problem of numerially estimating a convergence rate of the trace-class Markov operator in the second case.

\par\vspace{4mm}
\end{abstract}

\section{Introduction}
\label{sec:introduction}
Student-$t$ regression models are widely used for robust Bayesian inference. 
In order to obtain approximate posterior samples under Student-$t$ regression models, we usually take a data augmentation approach because the joint posterior distribution of the regression coefficients and error variance does not have a standard form and is difficult to directly sample from. 
In particular, we use the fact that Student's $t$-distribution is a scale mixture of normals and generate samples of mixing variables as well as the target parameters in order to approximate the posterior distribution. 

Efficiency of Markov chains based on such algorithms has been investigated by many authors. 
A popular criterion is geometric ergodicity. 
Sufficient conditions for this property were derived by \cite{hjkq2018} in the case where the prior distribution is improper and by \cite{bh2020} in the proper case. 
Recently, \cite{h2026} obtained conditions to cover the case of a Student-$t$ error density with few degrees of freedom. 

A stronger property of a Markov chain is that it is trace-class. 
There are several benefits related to this property (e.g., \cite{qh2018}). 
In particular, we can estimate the spectral gap of a trace-class Markov chain (\cite{qhk2019}), which is directly related to the geometric convergence rate of the chain. 
Although trace-class results for MCMC algorithms for regression models were obtained by \cite{jh2014} and \cite{qh2018}, their results are not applicable when we use error distributions with polynomial tails. 

In this paper, we obtain trace-class results for MCMC algorithms for Student-$t$ regression models. 
In Section \ref{sec:proper}, we consider the proper case where the prior distribution is normal-inverse gamma. 
In Section \ref{sec:improper}, we consider the improper case where the parameters follow the invariant distributions when the outcomes are univariate. 
We show that the Markov operator associated with a standard data augmentation algorithm is not trace-class and that the Markov operator associated with another algorithm is trace-class. 
In Section \ref{sec:improper_multi}, we consider the improper case when the outcomes are multivariate. 
We obtain a trace-class result for a parameter expanded data augmentation algorithm which is based on a univariate working parameter. 
In Section \ref{sec:sim}, as an example, we consider the estimation based on classical Monte Carlo samples of the spectral gap of the trace-class Markov operator of Section \ref{sec:improper}.

\section{The Case of a Proper Prior}
\label{sec:proper} 
In this section, we consider the multivariate linear regression model of \cite{bh2020}. 
We follow the notation of \cite{h2026}. 
Let $h \colon (0, \infty ) \to [0, \infty )$ be a normalized mixing density and suppose that 
\begin{align}
&\y _i \sim {1 \over | \bSi |^{1 / 2}} f_h ( \bSi ^{- 1 / 2} ( \y _i - \bBe ^{\top } \x _i )) \non 
\end{align}
for $i = 1, \dots , n$, where $\Y = ( \y _1 , \dots , \y _n )^{\top } \in \mathbb{R} ^{n \times d}$ and $\X = ( \x _1 , \dots , \x _n )^{\top } \in \mathbb{R} ^{n \times p}$ are outcome and explanatory variables, $\bBe \in \mathbb{R} ^{p \times d}$ and $\bSi > \O ^{(d)}$ are the matrix of regression coefficients and the covariance matrix, and where 
\begin{align}
f_h ( \bep ) &= \int_{0}^{\infty } {u^{d / 2} \over (2 \pi )^{d / 2}} \exp \Big( - {u \over 2} \| \bep \| ^2 \Big) h(u) du \text{,} \quad \bep \in \mathbb{R} ^d \text{,} \non 
\end{align}
is the error density. 
The prior distribution is given by 
\begin{align}
p( \bBe | \bSi ) p( \bSi ) &= {\rm{N}}_{p, d} ( \bBe | \B , \A , \bSi ) {\rm{IW}}_d ( \bSi | \nu , \C ) \text{,} \non 
\end{align}
where $\B \in \mathbb{R} ^{p \times d}$, $\A > \O ^{(p)}$, $\nu > d - 1$, and $\C > \O ^{(d)}$ are arbitrary; for the forms of the matrix normal and inverse Wishart densities, see Appendix A of \cite{bh2020}. 
The following data augmentation algorithm is from \cite{bh2020}. 

\begin{algo}
\label{algo:DA} 
The parameters $\bBe $ and $\bSi $ are updated in the following way. 
\begin{itemize}
\item
For each $i = 1, \dots , n$, sample $u_i \sim p( u_i | \bBe , \bSi , \Y )$, where 
\begin{align}
p( u_i | \bBe , \bSi , \y ) &\propto h( u_i ) {u_i}^{d / 2} \exp \{ - ( u_i / 2) ( \y _i - \bBe ^{\top } \x _i )^{\top } \bSi ^{- 1} ( \y _i - \bBe ^{\top } \x _i ) \} \text{.} \non 
\end{align}
Let $\U = \bdiag \u > \O ^{(n)}$. 
\item
Sample $\bSi \sim {\rm{IW}}_d (n + \nu , \bPsi ^{- 1} )$ and then $\bBe \sim {\rm{N}}_{p, d} ( \bGa , \bOm , \bSi )$, where 
\begin{align}
&\bPsi = \C ^{- 1} + \B ^{\top } \A ^{- 1} \B + \Y ^{\top } \U \Y - \bGa ^{\top } \bOm ^{- 1} \bGa \text{,} \non \\
&\bGa = ( \X ^{\top } \U \X + \A ^{- 1} )^{- 1} ( \X ^{\top } \U \Y + \A ^{- 1} \B ) \text{,} \quad \text{and} \non \\
&\bOm = ( \X ^{\top } \U \X + \A ^{- 1} )^{- 1} \text{.} \non 
\end{align}
\end{itemize}
\end{algo}

The following result is proved in Section \ref{sec:proof_prp_proper} of the Supplementary Material. 

\begin{thm}
\label{prp:proper} 
Let $a, b > 0$ and suppose that $h(u) = {\rm{Ga}} (u | a, b)$ for all $u \in (0, \infty )$. 
Then the Markov operator associated with Algorithm \ref{algo:DA} is trace-class. 
\end{thm}

\section{The Case of an Improper Prior and Univariate Responses}
\label{sec:improper} 
\subsection{The model and conditional distributions} 
\label{subsec:conditional_distributions} 
In this section, we consider the univariate Student-$t$ regression model under an improper prior. 
Specifically, we consider the model given by 
\begin{align}
&y_i \sim \int_{0}^{\infty } {\rm{N}} ( y_i | {\x _i}^{\top } \bbe , \sis / u_i ) {\rm{Ga}} ( u_i | a, b) d{u_i} \text{,} \quad i = 1, \dots , n \text{,} \non %
\end{align}
where $\y = ( y_i )_{i = 1}^{n} \in \mathbb{R} ^n$ and $\X = ( \x _1 , \dots , \x _n )^{\top } \in \mathbb{R} ^{n \times p}$ are the outcome and explanatory variables, $\u = ( u_i )_{i = 1}^{n} \in (0, \infty )^n$ are latent variables having the mixing distribution ${\rm{Ga}} (a, b)$ for $a, b > 0$, and $\bbe \in \mathbb{R} ^p$ and $\sis \in (0, \infty )$ are the regression coefficients and error variance. 
The prior is given by $( \bbe , \sis ) \sim 1 / \sis $. 

In the remainder of this section, we suppress the dependence on $\y $. 
The joint posterior distribution is given by 
\begin{align}
p( \bbe , \sis , \u ) &\propto {1 \over \sis } \prod_{i = 1}^{n} \Big[ {u_i}^{a - 1} e^{- b u_i} {{u_i}^{1 / 2} \over ( \sis )^{1 / 2}} \exp \Big\{ - {u_i \over 2 \sis } ( y_i - {\x _i}^{\top } \bbe )^2 \Big\} \Big] \text{.} \non 
\end{align}
Therefore, 
\begin{align}
p( \bbe , \sis | \u ) &\propto {1 \over ( \sis )^{1 + n / 2}} \exp \Big\{ - {1 \over 2 \sis } ( \y - \X \bbe )^{\top } \U ( \y - \X \bbe ) \Big\} \non \\
&\propto {\rm{IG}} \Big( \sis \Big| {n - p \over 2}, {\y ^{\top } \U \y - \y ^{\top } \U \X ( \X ^{\top } \U \X )^{- 1} \X ^{\top } \U \y \over 2} \Big) \non \\
&\quad \times {\rm{N}}_p ( \bbe | ( \X ^{\top } \U \X )^{- 1} \X ^{\top } \U \y , \sis ( \X ^{\top } \U \X )^{- 1} ) \text{,} \non 
\end{align}
where the second line follows since 
\begin{align}
( \y - \X \bbe )^{\top } \U ( \y - \X \bbe ) &= \bbe ^{\top } \X ^{\top } \U \X \bbe + \y ^{\top } \U \y - 2 \y ^{\top } \U \X \bbe \non \\
&= \{ \bbe - ( \X ^{\top } \U \X )^{- 1} \X ^{\top } \U \y \} ^{\top } \X ^{\top } \U \X \{ \bbe - ( \X ^{\top } \U \X )^{- 1} \X ^{\top } \U \y \} \non \\
&\quad + \y ^{\top } \U \y - \y ^{\top } \U \X ( \X ^{\top } \U \X )^{- 1} \X ^{\top } \U \y \text{.} \non 
\end{align}
Meanwhile, 
\begin{align}
p( \u , \sis | \bbe ) &\propto {1 \over ( \sis )^{1 + n / 2}} \prod_{i = 1}^{n} \Big( {u_i}^{1 / 2 + a - 1} \exp \Big[ - u_i \Big\{ {( y_i - {\x _i}^{\top } \bbe )^2 \over 2 \sis } + b \Big\} \Big] \Big) \non \\
&\propto \Big[ {1 \over ( \sis )^{1 + n / 2}} \prod_{i = 1}^{n} {1 \over \{ ( y_i - {\x _i}^{\top } \bbe )^2 / (2 \sis ) + b \} ^{1 / 2 + a}} \Big] \prod_{i = 1}^{n} {\rm{Ga}} \Big( u_i \Big| {1 \over 2} + a, {( y_i - {\x _i}^{\top } \bbe )^2 \over 2 \sis } + b \Big) \text{.} \non 
\end{align}

\subsection{An algorithm which is not trace class} 
\label{subsec:not_trace_class} 
The following algorithm is a standard data augmentation algorithm for the Student-$t$ regression model. 
\begin{algo}
\label{algo:not_trace_class} 
The parameters $\bbe $, $\sis $, and $\u $ are updated in the following way. 
\begin{itemize}
\item
Sample $( \bbe , \sis ) \sim p( \bbe , \sis | \u )$ by 
\begin{itemize}
\item
first sampling 
\begin{align}
\sis &\sim {\rm{IG}} \Big( \sis \Big| {n - p \over 2}, {\y ^{\top } \U \y - \y ^{\top } \U \X ( \X ^{\top } \U \X )^{- 1} \X ^{\top } \U \y \over 2} \Big) \non 
\end{align}
\item
and then sampling 
\begin{align}
\bbe &\sim {\rm{N}}_p ( \bbe | ( \X ^{\top } \U \X )^{- 1} \X ^{\top } \U \y , \sis ( \X ^{\top } \U \X )^{- 1} ) \text{.} \non 
\end{align}
\end{itemize}
\item
Sample $\u \sim p( \u | \bbe , \sis )$ by 
\begin{itemize}
\item
sampling 
\begin{align}
\u &\sim \prod_{i = 1}^{n} {\rm{Ga}} \Big( u_i \Big| {1 \over 2} + a, {( y_i - {\x _i}^{\top } \bbe )^2 \over 2 \sis } + b \Big) \text{.} \non 
\end{align}
\end{itemize}
\end{itemize}
\end{algo}

Although geometric ergodicity can be shown under some assumptions (see, for example, \cite{h2026}), we have the following result. 

\begin{thm}
\label{prp:not_trace_class} 
The Markov operator associated with Algorithm \ref{algo:not_trace_class} is not trace-class. 
\end{thm}

\noindent
{\bf Proof
.} \ \ By Theorem 2 of \cite{qhk2019}, the data augmentation Markov chain is %
trace-class if and only if $I < \infty $, where 
\begin{align}
I &= \int_{\mathbb{R} ^p \times (0, \infty ) \times (0, \infty )^n} p( \bbe , \sis | \u ) p( \u | \bbe , \sis ) d( \bbe , \sis , \u ) \text{.} \non 
\end{align}
The set $\bigcup_{i = 1}^{n} \{ \bbe \in \mathbb{R} ^p | | y_i - {\x _i}^{\top } \bbe | = 0 \} $ is closed. 
Therefore, there exist $\bbe _0 \in \mathbb{R} ^p$ and $\de , \ep > 0$ such that for all $\bbe \in U_{\de } ( \bbe _0 )$ and all $i = 1, \dots , n$, we have $| y_i - {\x _i}^{\top } \bbe | > \ep $. 
Let $M = \sup_{\bbe \in U_{\de } ( \bbe _0 )} \| \y - \X \bbe \| \in (0, \infty )$. 
Then for all $\bbe \in U_{\de } ( \bbe _0 )$, we have, by Section \ref{subsec:conditional_distributions}, 
\begin{align}
p( \bbe , \sis | \u ) &= \frac{ \displaystyle {1 \over ( \sis )^{1 + n / 2}} \exp \Big\{ - {1 \over 2 \sis } ( \y - \X \bbe )^{\top } \U ( \y - \X \bbe ) \Big\} }{ \displaystyle \int_{(0, \infty )} {(2 \pi )^{p / 2} / | \X ^{\top } \U \X |^{1 / 2} \over ( \sis )^{1 + (n - p) / 2}} \exp \Big\{ - {\y ^{\top } \U \y - \y ^{\top } \U \X ( \X ^{\top } \U \X )^{- 1} \X ^{\top } \U \y \over 2 \sis } \Big\} d\sis } \non \\
&\ge \frac{ \displaystyle {1 \over ( \sis )^{1 + n / 2}} \prod_{i = 1}^{n} \exp \Big( - {M^2 \over 2 \sis } u_i \Big) }{ \displaystyle {(2 \pi )^{p / 2} \over | \X ^{\top } \U \X |^{1 / 2}} {\Ga ((n - p) / 2) 2^{(n - p) / 2} \over \{ \y ^{\top } \U \y - \y ^{\top } \U \X ( \X ^{\top } \U \X )^{- 1} \X ^{\top } \U \y \} ^{(n - p) / 2}} } \non 
\end{align}
and 
\begin{align}
p( \u | \bbe , \sis ) &= \prod_{i = 1}^{n} \frac{ \displaystyle {u_i}^{1 / 2 + a - 1} \exp \Big[ - u_i \Big\{ {( y_i - {\x _i}^{\top } \bbe )^2 \over 2 \sis } + b \Big\} \Big] }{ \displaystyle \int_{(0, \infty )^n} {u_i}^{1 / 2 + a - 1} \exp \Big[ - u_i \Big\{ {( y_i - {\x _i}^{\top } \bbe )^2 \over 2 \sis } + b \Big\} \Big] d{u_i} } \non \\
&\ge \prod_{i = 1}^{n} \frac{ \displaystyle {u_i}^{1 / 2 + a - 1} \exp \Big\{ - u_i \Big( {M^2 \over 2 \sis } + b \Big) \Big\} }{ \displaystyle \int_{(0, \infty )^n} {u_i}^{1 / 2 + a - 1} \exp \Big\{ - u_i \Big( {\ep ^2 \over 2 \sis } + b \Big) \Big\} d{u_i} } \text{.} \non 
\end{align}
Thus, 
\begin{align}
I &\ge {1 \over M_1} \int_{U_{\de } ( \bbe _0 ) \times (0, \infty ) \times (0, \infty )^n} \Big( {| \X ^{\top } \U \X |^{1 / 2} \{ \y ^{\top } \U \y - \y ^{\top } \U \X ( \X ^{\top } \U \X )^{- 1} \X ^{\top } \U \y \} ^{(n - p) / 2} \over ( \sis )^{1 + n / 2}} \non \\
&\quad \times \Big[ \prod_{i = 1}^{n} \frac{ \displaystyle {u_i}^{1 / 2 + a - 1} \exp \{ - u_i ( M^2 / \sis + b) \} }{ \displaystyle \Ga (1 / 2 + a) / \{ \ep ^2 / (2 \sis ) + b \} ^{1 / 2 + a} } \Big] \Big) d( \bbe , \sis , \u ) \non \\
&= {1 \over M_1} \Big\{ \int_{U_{\de } ( \bbe _0 )} 1 d\bbe \Big\} \int_{(0, \infty ) \times (0, \infty )^n} \Big( {| \X ^{\top } \V \X |^{1 / 2} \{ \y ^{\top } \V \y - \y ^{\top } \V \X ( \X ^{\top } \V \X )^{- 1} \X ^{\top } \V \y \} ^{(n - p) / 2} \over ( M^2 / \sis + b)^{p / 2 + (n - p) / 2} ( \sis )^{1 + n / 2}} \non \\
&\quad \times \Big[ \prod_{i = 1}^{n} \frac{ \displaystyle {v_i}^{1 / 2 + a - 1} \exp (- v_i ) / \Ga (1 / 2 + a) }{ \displaystyle ( M^2 / \sis + b)^{1 / 2 + a} / \{ \ep ^2 / (2 \sis ) + b \} ^{1 / 2 + a} } \Big] \Big) d( \sis , \v ) \non 
\end{align}
for some $M_1 > 0$ and the right-hand side of the above inequality is infinite since 
\begin{align}
&\int_{0}^{1} {1 \over ( M^2 / \sis + b)^{p / 2 + (n - p) / 2} ( \sis )^{1 + n / 2}} {1 \over [( M^2 / \sis + b)^{1 / 2 + a} / \{ \ep ^2 / (2 \sis ) + b \} ^{1 / 2 + a} ]^n} d\sis %
= \infty \text{.} \non 
\end{align}
This completes the proof. 
\hfill$\Box$

\subsection{An algorithm which is trace class} 
\label{subsec:trace_class} 
In the following collapsed Gibbs algorithm, $\u $ is sampled unconditionally on $\sis $. 

\begin{algo}
\label{algo:trace_class} 
The parameters $\bbe $, $\sis $, and $\u $ are updated in the following way. 
\begin{itemize}
\item
Sample $( \bbe , \sis ) \sim p( \bbe , \sis | \u )$ by 
\begin{itemize}
\item
first sampling 
\begin{align}
\sis &\sim {\rm{IG}} \Big( \sis \Big| {n - p \over 2}, {\y ^{\top } \U \y - \y ^{\top } \U \X ( \X ^{\top } \U \X )^{- 1} \X ^{\top } \U \y \over 2} \Big) \non 
\end{align}
\item
and then sampling 
\begin{align}
\bbe &\sim {\rm{N}}_p ( \bbe | ( \X ^{\top } \U \X )^{- 1} \X ^{\top } \U \y , \sis ( \X ^{\top } \U \X )^{- 1} ) \text{.} \non 
\end{align}
\end{itemize}
\item
Sample $( \u , \sis ) \sim p( \u , \sis | \bbe )$ by 
\begin{itemize}
\item
first sampling 
\begin{align}
\sis &\sim p( \sis | \bbe ) \propto ( \sis )^{n a - 1} \prod_{i = 1}^{n} {1 \over \{ ( y_i - {\x _i}^{\top } \bbe )^2 / (2 b) + \sis \} ^{1 / 2 + a}} \non 
\end{align}
\item
and then sampling 
\begin{align}
\u &\sim \prod_{i = 1}^{n} {\rm{Ga}} \Big( u_i \Big| {1 \over 2} + a, {( y_i - {\x _i}^{\top } \bbe )^2 \over 2 \sis } + b \Big) \text{.} \non 
\end{align}
\end{itemize}
\end{itemize}
\end{algo}

The conditional density of $\log \sis $ given $\bbe $ is log-concave. 
Therefore, we can use the method of \cite{gw1992}.

When we use Algorithm \ref{algo:trace_class}, we sample $\bbe \sim p( \bbe | \u )$ and $\u \sim p( \u | \bbe )$. 
Therefore, this is a data augmentation algorithm based on the joint density $p( \bbe , \u )$. 
The transition density is given by 
\begin{align}
k( \bbe _{\rm{new}} | \bbe _{\rm{old}} ) &= \int_{(0, \infty )^n} p( \bbe _{\rm{new}} | \u ) p( \u | \bbe _{\rm{old}} ) d\u \text{,} \quad \bbe _{\rm{old}} , \bbe _{\rm{new}} \in \mathbb{R} ^p \text{.} \non 
\end{align}

\begin{thm}
\label{prp:trace_class} 
Suppose that $n \ge p + 1$. 
Suppose that $1 / 2 + a > n / (n - p)$. 
Then $k$ is trace-class. 
\end{thm}

The lower bound on $a$ does not increase linearly with $n$. 
This is in contrast for example to the case considered by \cite{rh2010}.

\bigskip

\noindent
{\bf Proof of Theorem \ref{prp:trace_class}.} \ \ By Theorem 2 of \cite{qhk2019}, the data augmentation Markov chain is %
trace-class if and only if $I < \infty $, where 
\begin{align}
I &= \int_{\mathbb{R} ^p} k( \bbe | \bbe ) d\bbe = \int_{\mathbb{R} ^p \times (0, \infty )^n} p( \bbe | \u ) p( \u | \bbe ) d( \bbe , \u ) \text{.} \non 
\end{align}
By Section \ref{subsec:conditional_distributions}, 
\begin{align}
p( \bbe | \u ) %
&= \frac{ \displaystyle \int_{0}^{\infty } {1 \over ( \sis )^{1 + n / 2}} \exp \Big\{ - {1 \over 2 \sis } ( \y - \X \bbe )^{\top } \U ( \y - \X \bbe ) \Big\} d\sis }{ \displaystyle \int_{\mathbb{R} ^p \times (0, \infty )} {1 \over ( \sis )^{1 + n / 2}} \exp \Big\{ - {1 \over 2 \sis } ( \y - \X \bbe )^{\top } \U ( \y - \X \bbe ) \Big\} d( \bbe , \sis ) } \non \\
&= {\Ga (n / 2) / \{ \pi ^{p / 2} \Ga ((n - p) / 2) \} \over \{ ( \y - \X \bbe )^{\top } \U ( \y - \X \bbe ) \} ^{n / 2}} | \X ^{\top } \U \X |^{1 / 2} \{ \y ^{\top } \U \y - \y ^{\top } \U \X ( \X ^{\top } \U \X )^{- 1} \X ^{\top } \U \y \} ^{(n - p) / 2} \label{eq:p_be_u} 
\end{align}
and 
\begin{align}
p( \u | \bbe ) &= \frac{ \displaystyle \int_{0}^{\infty } {1 \over ( \sis )^{1 + n / 2}} \prod_{i = 1}^{n} \Big( {u_i}^{1 / 2 + a - 1} \exp \Big[ - u_i \Big\{ {( y_i - {\x _i}^{\top } \bbe )^2 \over 2 \sis } + b \Big\} \Big] \Big) d\sis }{ \displaystyle \int_{(0, \infty )^n \times (0, \infty )} {1 \over ( \sis )^{1 + n / 2}} \prod_{i = 1}^{n} \Big( {u_i}^{1 / 2 + a - 1} \exp \Big[ - u_i \Big\{ {( y_i - {\x _i}^{\top } \bbe )^2 \over 2 \sis } + b \Big\} \Big] \Big) d( \u , \sis ) } \non \\
&= {2^{n / 2} b^{n a + n / 2} \Ga (n / 2) \over \{ \Ga (1 / 2 + a) \} ^n} {\prod_{i = 1}^{n} ( {u_i}^{1 / 2 + a - 1} e^{- b u_i} ) \over \big\{ \sum_{i = 1}^{n} u_i ( y_i - {\x _i}^{\top } \bbe )^2 \big\} ^{n / 2}} \non \\
&\quad / \int_{0}^{\infty } ( \sis )^{n a - 1} \Big[ \prod_{i = 1}^{n} {1 \over \{ ( y_i - {\x _i}^{\top } \bbe )^2 / (2 b) + \sis \} ^{1 / 2 + a}} \Big] d\sis \text{.} \non 
\end{align}
Therefore, 
\begin{align}
p( \bbe | \u ) p( \u | \bbe ) &\le M_1 {| \X ^{\top } \U \X |^{1 / 2} ( \y ^{\top } \U \y )^{(n - p) / 2} \over \big\{ \sum_{i = 1}^{n} u_i ( y_i - {\x _i}^{\top } \bbe )^2 \big\} ^n} \Big\{ \prod_{i = 1}^{n} ( {u_i}^{1 / 2 + a - 1} e^{- b u_i} ) \Big\} \non \\
&\quad / \int_{0}^{\infty } ( \sis )^{n a - 1} \Big[ \prod_{i = 1}^{n} {1 \over \{ ( y_i - {\x _i}^{\top } \bbe )^2 / (2 b) + \sis \} ^{1 / 2 + a}} \Big] d\sis \non 
\end{align}
for some $M_1 > 0$. 
Note that 
\begin{align}
| \X ^{\top } \U \X |^{1 / 2} ( \y ^{\top } \U \y )^{(n - p) / 2} \prod_{i = 1}^{n} ( {u_i}^{1 / 2 + a - 1} e^{- b u_i} ) &\le M_2 \prod_{i = 1}^{n} ( {u_i}^{1 / 2 + a - 1} e^{- b u_i / 2} ) \non 
\end{align}
for some $M_2 > 0$ and that 
\begin{align}
\int_{0}^{\infty } ( \sis )^{n a - 1} \Big[ \prod_{i = 1}^{n} {1 \over \{ ( y_i - {\x _i}^{\top } \bbe )^2 / (2 b) + \sis \} ^{1 / 2 + a}} \Big] d\sis &\ge {1 \over M_3} \int_{0}^{\infty } {( \sis )^{n a - 1} \over \prod_{i = 1}^{n} (1 + \| \bbe \| ^2 + \sis )^{1 / 2 + a}} d\sis \non \\
&= {B(n a, n / 2) \over M_3} {1 \over (1 + \| \bbe \| ^2 )^{n / 2}} \non 
\end{align}
for some $M_3 > 0$. 
Then 
\begin{align}
p( \bbe | \u ) p( \u | \bbe ) &\le M_4 \gt ( \bbe , \u ) \non 
\end{align}
for some $M_4 > 0$, where 
\begin{align}
\gt ( \bbe , \u ) &= {(1 + \| \bbe \| ^2 )^{n / 2} \over \big\{ \sum_{i = 1}^{n} u_i ( y_i - {\x _i}^{\top } \bbe )^2 \big\} ^n} \prod_{i = 1}^{n} ( {u_i}^{1 / 2 + a - 1} e^{- b u_i / 2} ) \text{.} \non 
\end{align}
Since 
\begin{align}
I &\le M_4 \int_{\mathbb{R} ^p \times (0, \infty )^n} \gt ( \bbe , \u ) d( \bbe , \u ) = M_4 \sum_{\ta \in \mathfrak{S} _n} \int_{\mathbb{R} ^p \times (0, \infty )^n} 1( u_{\ta (1)} \le \dots \le u_{\ta (n)} ) \gt ( \bbe , \u ) d( \bbe , \u ) \text{,} \non 
\end{align}
where $\mathfrak{S} _n$ is the set of permutations, it is sufficient to show that $I' < \infty $, where 
\begin{align}
I' &= \int_{\mathbb{R} ^p \times (0, \infty )^n} 1( u_1 \le \dots \le u_n ) \gt ( \bbe , \u ) d( \bbe , \u ) \text{.} \non 
\end{align}

We have 
\begin{align}
{1( u_1 \le \dots \le u_n ) \over \big\{ \sum_{i = 1}^{n} u_i ( y_i - {\x _i}^{\top } \bbe )^2 \big\} ^n} &= \prod_{i = 1}^{n - p} {1( u_1 \le \dots \le u_n ) \over \big\{ \sum_{i' = 1}^{n} u_{i'} ( y_{i'} - {\x _{i'}}^{\top } \bbe )^2 \big\} ^{n / (n - p)}} \non \\
&\le \prod_{i = 1}^{n - p} {1 \over \big\{ u_i \sum_{i' = i}^{i + p} ( y_{i'} - {\x _{i'}}^{\top } \bbe )^2 \big\} ^{n / (n - p)}} \text{.} \non 
\end{align}
Since 
\begin{align}
&\sum_{i' = i}^{i + p} ( y_{i'} - {\x _{i'}}^{\top } \bbe )^2 \non \\
&= \bbe ^{\top } \{ \X (i) \} ^{\top } \X (i) \bbe - 2 [ \{ \X (i) \} ^{\top } \y (i)]^{\top } \bbe + \| \y (i) \| ^2 \non \\
&= ( \bbe - [ \{ \X (i) \} ^{\top } \X (i)]^{- 1} \{ \X (i) \} ^{\top } \y (i))^{\top } \{ \X (i) \} ^{\top } \X (i) ( \bbe - [ \{ \X (i) \} ^{\top } \X (i)]^{- 1} \{ \X (i) \} ^{\top } \y (i)) \non \\
&\quad + \{ \y (i) \} ^{\top } ( \I ^{(p + 1)} - \X (i) [ \{ \X (i) \} ^{\top } \X (i)]^{- 1} \{ \X (i) \} ^{\top } ) \y (i) \text{,} \non 
\end{align}
where $\y (i) = ( y_{i'} )_{i' = i}^{i + p} \in \mathbb{R} ^{p + 1}$ and $\X (i) = ( \x _i , \dots , \x _{i + p} )^{\top } \in \mathbb{R} ^{(p + 1) \times p}$, for all $i = 1, \dots , n - p$, it follows that 
\begin{align}
{1( u_1 \le \dots \le u_n ) \over \big\{ \sum_{i = 1}^{n} u_i ( y_i - {\x _i}^{\top } \bbe )^2 \big\} ^n} &\le M_5 \prod_{i = 1}^{n - p} {1 / {u_i}^{n / (n - p)} \over \{ 1 + \| \bbe - \b (i) \| ^2 \} ^{n / (n - p)}} \non \\
&\le M_5 \prod_{i = 1}^{n - p} \Big[ {\{ 1 + 2 \| \bbe - \b (i) \| ^2 + 2 \| \b (i) \| ^2 \} ^{n / (n - p)} \over \{ 1 + \| \bbe - \b (i) \| ^2 \} ^{n / (n - p)}} {1 / {u_i}^{n / (n - p)} \over (1 + \| \bbe \| ^2 )^{n / (n - p)}} \Big] \non \\
&\le M_6 \prod_{i = 1}^{n - p} {1 + 1 / {u_i}^{n / (n - p)} \over (1 + \| \bbe \| ^2 )^{n / (n - p)}} \non 
\end{align}
for some $M_5 , M_6 > 0$, where $\b (i) = [ \{ \X (i) \} ^{\top } \X (i)]^{- 1} \{ \X (i) \} ^{\top } \y (i)$ for $i = 1, \dots , n - p$. 
Thus, 
\begin{align}
1( u_1 \le \dots \le u_n ) \gt ( \bbe , \u ) &\le M_6 {(1 + \| \bbe \| ^2 )^{n / 2} \over (1 + \| \bbe \| ^2 )^n} \prod_{i = 1}^{n} \Big[ \Big\{ 1 + {1 \over {u_i}^{n / (n - p)}} \Big\} {u_i}^{1 / 2 + a - 1} e^{- b u_i / 2} \Big] \non 
\end{align}
and 
\begin{align}
I' &\le M_7 \int_{\mathbb{R} ^p} {1 \over (1 + \| \bbe \| ^2 )^{n / 2}} d\bbe < \infty \non 
\end{align}
for some $M_7 > 0$. 
This completes the proof. 
\hfill$\Box$

\section{The Case of an Improper Prior and Multivariate Responses}
\label{sec:improper_multi} 
\subsection{The model and conditional distributions} 
\label{subsec:conditional_distributions_multi} 
In this section, we consider the multivariate Student-$t$ regression model under an improper prior. 
In this case, the idea of collpsing used in the previous section is not useful since it is difficult to sample the multivariate covariance matrix from its conditional distribution conditional only on the regression coefficients. 
However, we will see that we can construct an efficient and easy-to-implement parameter expanded data augmentation (\cite{lw1999}) algorithm by introducing a univariate working parameter appropriately. 

We consider the model of \cite{hjkq2018} when the mixing density is gamma. 
Suppose that 
\begin{align}
&\y _i \sim {\rm{N}}_d ( \y _i | \bBe ^{\top } \x _i , (1 / u_i ) \bSi ) \text{,} \quad i = 1, \dots , n \text{,} \non \\
&u_i \sim {\rm{Ga}} ( u_i | a, b) \text{,} \quad i = 1, \dots , n \text{,} \non \\
&( \bBe , \bSi ) \sim 1 / | \bSi |^c \text{,} \non 
\end{align}
where $\Y = ( \y _1 , \dots , \y _n )^{\top } \in \mathbb{R} ^{n \times d}$ and $\X = ( \x _1 , \dots , \x _n )^{\top } \in \mathbb{R} ^{n \times p}$ are outcome and explanatory variables while $\bBe \in \mathbb{R} ^{p \times d}$ and $\bSi > \O ^{(d)}$ are the matrix of regression coefficients and the covariance matrix. 
Here, $\u = ( u_i )_{i = 1}^{n}$ are the mixing variables and 
\begin{align}
p( \y _i | \bBe , \bSi ) %
&= {\Ga (d / 2 + a) b^a \over (2 \pi )^{d / 2} \Ga (a)} {1 \over | \bSi |^{d / 2}} {1 \over \{ b + ( \y _i - \bBe ^{\top } \x _i )^{\top } \bSi ^{- 1} ( \y _i - \bBe ^{\top } \x _i ) / 2 + b \} ^{d / 2 + a}} \text{.} \non 
\end{align}
We treat the case $c \ge d - (n - p - 1) / 2$ and we are primarily interested in the special case $c = (d + 1) / 2$. 
We assume posterior propriety; for some conditions for posterior propriety, see, for example, Proposition A1 of \cite{h2026}. 
In the remainder of this section, we suppress the dependence on $\Y $. 

In order to construct a parameter expanded data augmentation algorithm, we let 
\begin{align}
&\al \sim {\rm{Ga}} ( \al | e, f) \non 
\end{align}
for $e, f > 0$ and use the transformation 
\begin{align}
&\bSi = \bSit / \al \text{.} \non 
\end{align}
Since 
\begin{align}
p( \bBe , \bSi , \u ) &\propto {1 \over | \bSi |^{n / 2 + c}} \prod_{i = 1}^{n} \Big[ {u_i}^{d / 2 + a - 1} e^{- b u_i} \exp \Big\{ - {u_i \over 2} ( \y _i - \bBe ^{\top } \x _i )^{\top } \bSi ^{- 1} ( \y _i - \bBe ^{\top } \x _i ) \Big\} \Big] \text{,} \non 
\end{align}
the joint density of $( \bBe , \bSit , \u , \al )$ is 
\begin{align}
p( \bBe , \bSit , \u , \al ) &\propto p( \al ) {\al ^{d (n / 2 + c) - d (d + 1) / 2} \over | \bSit |^{n / 2 + c}} \prod_{i = 1}^{n} \Big[ {u_i}^{d / 2 + a - 1} e^{- b u_i} \exp \Big\{ - {\al \over 2} u_i ( \y _i - \bBe ^{\top } \x _i )^{\top } \bSit ^{- 1} ( \y _i - \bBe ^{\top } \x _i ) \Big\} \Big] \text{.} \non 
\end{align}

First, 
\begin{align}
p( \bBe , \bSit , \al | \u ) &\propto p( \al ) {\al ^{d (n / 2 + c) - d (d + 1) / 2} \over | \bSit |^{n / 2 + c}} \prod_{i = 1}^{n} \Big[ \exp \Big\{ - {\al \over 2} u_i ( \y _i - \bBe ^{\top } \x _i )^{\top } \bSit ^{- 1} ( \y _i - \bBe ^{\top } \x _i ) \Big\} \Big] \text{.} \non 
\end{align}
Since 
\begin{align}
\sum_{i = 1}^{n} u_i ( \y _i - \bBe ^{\top } \x _i )^{\top } \bSit ^{- 1} ( \y _i - \bBe ^{\top } \x _i ) &= \tr \{ \bSit ^{- 1} ( \Y - \X \bBe )^{\top } \U ( \Y - \X \bBe ) \} \non 
\end{align}
where $\U = \bdiag \u $ and since 
\begin{align}
( \Y - \X \bBe )^{\top } \U ( \Y - \X \bBe ) &= \{ \bBe - ( \X ^{\top } \U \X )^{- 1} \X ^{\top } \U \Y \} ^{\top } \X ^{\top } \U \X \{ \bBe - ( \X ^{\top } \U \X )^{- 1} \X ^{\top } \U \Y \} \non \\
&\quad + \Y ^{\top } \U \Y - \Y ^{\top } \U \X ( \X ^{\top } \U \X )^{- 1} \X ^{\top } \U \Y \text{,} \non 
\end{align}
it follows that 
\begin{align}
p( \bBe , \bSit , \al | \u ) &\propto p( \al ) {\al ^{d (n / 2 + c) - d (d + 1) / 2} \over | \bSit |^{n / 2 + c}} \exp \Big( - {\al \over 2} \tr [ \bSit ^{- 1} \{ \Y ^{\top } \U \Y - \Y ^{\top } \U \X ( \X ^{\top } \U \X )^{- 1} \X ^{\top } \U \Y \} ] \Big) \non \\
&\quad \times \exp \Big( - {\al \over 2} \tr [ \bSit ^{- 1} \{ \bBe - ( \X ^{\top } \U \X )^{- 1} \X ^{\top } \U \Y \} ^{\top } \X ^{\top } \U \X \{ \bBe - ( \X ^{\top } \U \X )^{- 1} \X ^{\top } \U \Y \} ] \Big) \non \\
&\propto p( \al ) {\rm{IW}}_d ( \bSit | n - p - d - 1 + 2 c, \{ \Y ^{\top } \U \Y - \Y ^{\top } \U \X ( \X ^{\top } \U \X )^{- 1} \X ^{\top } \U \Y \} ^{- 1} / \al ) \non \\
&\quad \times {\rm{N}}_{p, d} ( \bBe | ( \X ^{\top } \U \X )^{- 1} \X ^{\top } \U \Y , ( \X ^{\top } \U \X )^{- 1} , \bSit / \al ) \text{.} \non 
\end{align}
Next, 
\begin{align}
p( \u , \al | \bBe , \bSit ) &\propto p( \al ) \al ^{d (n / 2 + c) - d (d + 1) / 2} \prod_{i = 1}^{n} \Big[ {u_i}^{d / 2 + a - 1} e^{- b u_i} \exp \Big\{ - {\al \over 2} u_i ( \y _i - \bBe ^{\top } \x _i )^{\top } \bSit ^{- 1} ( \y _i - \bBe ^{\top } \x _i ) \Big\} \Big] \non \\
&\propto p( \al ) \al ^{d (n / 2 + c) - d (d + 1) / 2} \prod_{i = 1}^{n} {{\rm{Ga}} ( u_i | d / 2 + a, b + \al ( \y _i - \bBe ^{\top } \x _i )^{\top } \bSit ^{- 1} ( \y _i - \bBe ^{\top } \x _i ) / 2) \over \{ b + \al ( \y _i - \bBe ^{\top } \x _i )^{\top } \bSit ^{- 1} ( \y _i - \bBe ^{\top } \x _i ) / 2 \} ^{d / 2 + a}} \text{.} \non 
\end{align}
Thus, we can use the following parameter expanded data augmentation algorithm. 

\begin{algo}
\label{algo:PX} 
The parameters $\bBe $, $\bSit $, $\u $, and $\al $ are updated in the following way. 
\begin{itemize}
\item
Sample $( \bBe , \bSit , \al ) \sim p( \bBe , \bSit , \al | \u )$ by 
\begin{itemize}
\item
first sampling 
\begin{align}
\al &\sim {\rm{Ga}} ( \al | e, f) \text{,} \non 
\end{align}
\item
second sampling 
\begin{align}
\bSit &\sim {\rm{IW}}_d ( \bSit | n - p - d - 1 + 2 c, \{ \Y ^{\top } \U \Y - \Y ^{\top } \U \X ( \X ^{\top } \U \X )^{- 1} \X ^{\top } \U \Y \} ^{- 1} / \al ) \text{,} \non 
\end{align}
\item
and third sampling 
\begin{align}
\bBe &\sim {\rm{N}}_{p, d} ( \bBe | ( \X ^{\top } \U \X )^{- 1} \X ^{\top } \U \Y , ( \X ^{\top } \U \X )^{- 1} , \bSit / \al ) \text{.} \non 
\end{align}
\end{itemize}
\item
Sample $( \u , \al ) \sim p( \u , \al | \bBe , \bSit )$ by 
\begin{itemize}
\item
first sampling 
\begin{align}
\al &\sim p( \al | \bBe , \bSit ) \non \\
&\propto {\al ^{d \{ n / 2 + c - (d + 1) / 2 \} + e - 1} \exp (- f \al ) \over \prod_{i = 1}^{n} \{ b + \al ( \y _i - \bBe ^{\top } \x _i )^{\top } \bSit ^{- 1} ( \y _i - \bBe ^{\top } \x _i ) / 2 \} ^{d / 2 + a}} \non 
\end{align}
\item
and then sampling 
\begin{align}
\u &\sim \prod_{i = 1}^{n} {\rm{Ga}} ( u_i | d / 2 + a, \al ( \y _i - \bBe ^{\top } \x _i )^{\top } \bSit ^{- 1} ( \y _i - \bBe ^{\top } \x _i ) / 2 + b) \text{.} \non 
\end{align}
\end{itemize}
\end{itemize}
\end{algo}

The conditional density of $( \log \al ) | ( \bBe , \bSit )$ is log-concave for example if $c = (d + 1) / 2$ and $e \ge 1$. 
In this case, we can use the method of \cite{gw1992}.

\subsection{Trace-class results} 
\label{subsec:trace-class_improper_multi} 
Algorithm \ref{algo:PX} is a data augmentation algorithm based on the joint density $p( \bBe , \bSit , \u )$. 
The transition density is given by 
\begin{align}
k( \bBe _{\rm{new}} , \bSit _{\rm{new}} | \bBe _{\rm{old}} , \bSit _{\rm{old}} ) &= \int_{(0, \infty )^n} p( \bBe _{\rm{new}} , \bSit _{\rm{new}} | \u ) p( \u | \bBe _{\rm{old}} , \bSit _{\rm{old}} ) d\u \text{,} \non 
\end{align}
$( \bBe _{\rm{new}} , \bSit _{\rm{new}} ), ( \bBe _{\rm{old}} , \bSit _{\rm{old}} ) \in \mathbb{R} ^{p \times d} \times \mathcal{S} ^{>} (d)$, where $\mathcal{S} ^{>} (d) = \{ \M \in \mathbb{R} ^{d \times d} | \M > \O ^{(d)} \} $. 
It follows from Theorem 2 of \cite{qhk2019} that the Markov operator based on this transition density is trace-class if and only if $I < \infty $, where 
\begin{align}
I &= \int_{\mathbb{R} ^{p \times d} \times \mathcal{S} ^{>} (d) \times (0, \infty )^n} p( \bBe , \bSit | \u ) p( \u | \bBe , \bSit ) d( \bBe , \bSit , \u ) \text{.} \non 
\end{align}
By Section \ref{subsec:conditional_distributions_multi}, 
\begin{align}
&p( \bBe , \bSit | \u ) \non \\
&= \int_{0}^{\infty } {\al ^{d \{ n / 2 + c - (d + 1) / 2 \} + e - 1} \over | \bSit |^{n / 2 + c}} e^{- f \al } \exp \Big\{ - {\al \over 2} \sum_{i = 1}^{n} u_i ( \y _i - \bBe ^{\top } \x _i )^{\top } \bSit ^{- 1} ( \y _i - \bBe ^{\top } \x _i ) \Big\} d\al \non \\
&\quad / \int_{0}^{\infty } \Big[ \int {\al ^{d \{ n / 2 + c - (d + 1) / 2 \} + e - 1} \over | \bSit |^{n / 2 + c}} e^{- f \al } \exp \Big\{ - {\al \over 2} \sum_{i = 1}^{n} u_i ( \y _i - \bBe ^{\top } \x _i )^{\top } \bSit ^{- 1} ( \y _i - \bBe ^{\top } \x _i ) \Big\} d( \bBe , \bSit ) \Big] d\al \non \\
&= {C_1 \over | \bSit |^{n / 2 + c}} {| \X ^{\top } \U \X |^{d / 2} | \Y ^{\top } \U \Y - \Y ^{\top } \U \X ( \X ^{\top } \U \X )^{- 1} \X ^{\top } \U \Y |^{(n - p - d - 1) / 2 + c} \over \big\{ f + \sum_{i = 1}^{n} u_i ( \y _i - \bBe ^{\top } \x _i )^{\top } \bSit ^{- 1} ( \y _i - \bBe ^{\top } \x _i ) / 2 \big\} ^{d \{ n / 2 + c - (d + 1) / 2 \} + e}} \label{improper_multi_1} 
\end{align}
for some constant $C_1 > 0$ and 
\begin{align}
&p( \u | \bBe , \bSit ) \non \\
&= \int_{0}^{\infty } \al ^{d \{ n / 2 + c - (d + 1) / 2 \} + e - 1} e^{- f \al } \Big[ \prod_{i = 1}^{n} {{u_i}^{d / 2 + a - 1} e^{- b u_i} \over \exp \{ ( \al / 2) u_i ( \y _i - \bBe ^{\top } \x _i )^{\top } \bSit ^{- 1} ( \y _i - \bBe ^{\top } \x _i ) \} } \Big] d\al \non \\
&\quad / \int_{0}^{\infty } \al ^{d \{ n / 2 + c - (d + 1) / 2 \} + e - 1} e^{- f \al } \Big[ \prod_{i = 1}^{n} \int_{0}^{\infty } {{u_i}^{d / 2 + a - 1} e^{- b u_i} \over \exp \{ ( \al / 2) u_i ( \y _i - \bBe ^{\top } \x _i )^{\top } \bSit ^{- 1} ( \y _i - \bBe ^{\top } \x _i ) \} } d{u_i} \Big] d\al \non \\
&= C_2 \frac{ \big\{ \prod_{i = 1}^{n} ( {u_i}^{d / 2 + a - 1} e^{- b u_i} ) \big\} / \big\{ f + \sum_{i = 1}^{n} u_i ( \y _i - \bBe ^{\top } \x _i )^{\top } \bSit ^{- 1} ( \y _i - \bBe ^{\top } \x _i ) / 2 \big\} ^{d \{ n / 2 + c - (d + 1) / 2 \} + e} }{ \int_{0}^{\infty } \big[ \al ^{d \{ n / 2 + c - (d + 1) / 2 \} + e - 1} e^{- f \al } / \prod_{i = 1}^{n} \{ b + \al ( \y _i - \bBe ^{\top } \x _i )^{\top } \bSit ^{- 1} ( \y _i - \bBe ^{\top } \x _i ) / 2 \} ^{d / 2 + a} \big] d\al } \text{.} \label{improper_multi_2} 
\end{align}
for some constant $C_2 > 0$. 

The following result is the main result of this section. 

\begin{thm}
\label{prp:trace-class_improper_multi} 
Assume that $n \ge p + d + 1$, that $n - p > 2 (d - c)$, and that $e > d (d - 1) / 2$. 
Suppose that 
\begin{align}
&{1 \over 2} + {a \over d} > {n + 2 c - (d + 1) + 2 e / d \over n + 1 - p - d} \text{.} \label{eq:condition_improper_multi} 
\end{align}
Then the Markov operator associated with Algorithm \ref{algo:PX} is trace-class. 
\end{thm}

If we set $e = d (d + 1) / 2$, (\ref{eq:condition_improper_multi}) becomes 
\begin{align}
&{1 \over 2} + {a \over d} > {n + 2 c \over n + 1 - p - d} \text{.} \non 
\end{align}
Therefore, if $a / d > 1 / 2$, the conclusion of the theorem holds for $e = d (d + 1) / 2$ for sufficiently large $n$. 

\bigskip

\noindent
{\bf Proof of Theorem \ref{prp:trace-class_improper_multi}.} \ \ We show that 
\begin{align}
I' &= \int_{\mathbb{R} ^{p \times d} \times \mathcal{S} ^{>} (d) \times (0, \infty )^n} 1( u_1 \le \dots \le u_n ) p( \bBe , \bSit | \u ) p( \u | \bBe , \bSit ) d( \bBe , \bSit , \u ) < \infty \text{.} \non 
\end{align}
By (\ref{improper_multi_1}) and (\ref{improper_multi_2}), 
\begin{align}
&1( u_1 \le \dots \le u_n ) p( \bBe , \bSit | \u ) p( \u | \bBe , \bSit ) \non \\
&\le M_1 {1 \over | \bSit |^{n / 2 + c}} {| \X ^{\top } \U \X |^{d / 2} | \Y ^{\top } \U \Y - \Y ^{\top } \U \X ( \X ^{\top } \U \X )^{- 1} \X ^{\top } \U \Y |^{(n - p - d - 1) / 2 + c} \over \big\{ f + \sum_{i = 1}^{n} u_i ( \y _i - \bBe ^{\top } \x _i )^{\top } \bSit ^{- 1} ( \y _i - \bBe ^{\top } \x _i ) / 2 \big\} ^{2 [d \{ n / 2 + c - (d + 1) / 2 \} + e]}} \non \\
&\quad \times \frac{ \prod_{i = 1}^{n} ( {u_i}^{d / 2 + a - 1} e^{- b u_i} ) }{ \int_{0}^{\infty } \big[ \al ^{d \{ n / 2 + c - (d + 1) / 2 \} + e - 1} e^{- f \al } / \prod_{i = 1}^{n} \{ b + \al ( \y _i - \bBe ^{\top } \x _i )^{\top } \bSit ^{- 1} ( \y _i - \bBe ^{\top } \x _i ) / 2 \} ^{d / 2 + a} \big] d\al } \non \\
&\le M_2 {1 \over | \bSit |^{n / 2 + c}} {| \X ^{\top } \U \X |^{d / 2} | \Y ^{\top } \U \Y - \Y ^{\top } \U \X ( \X ^{\top } \U \X )^{- 1} \X ^{\top } \U \Y |^{(n - p - d - 1) / 2 + c} \over \{ 1 + \tr ( \bSit ^{- 1} ) + \tr ( \bBe \bSit ^{- 1} \bBe ^{\top } ) \} ^{2 [d \{ n / 2 + c - (d + 1) / 2 \} + e]}} \non \\
&\quad \times \frac{ \prod_{i = 1}^{n} ((1 + 1 / u_i )^{2 [d \{ n / 2 + c - (d + 1) / 2 \} + e] / (n + 1 - p - d)} {u_i}^{d / 2 + a - 1} e^{- b u_i} ) }{ \int_{0}^{\infty } \big( \al ^{d \{ n / 2 + c - (d + 1) / 2 \} + e - 1} e^{- f \al } / \prod_{i = 1}^{n} [b + M \al \{ \tr ( \bSit ^{- 1} ) + \tr ( \bBe \bSit ^{- 1} \bBe ^{\top } ) \} / 2]^{d / 2 + a} \big) d\al } \non 
\end{align}
for some $M_1 , M_2 > 0$, where the second inequality follows from part (ii) of Lemma \ref{lem:numerator} and Lemma \ref{lem:denominator}. 
Since 
\begin{align}
&\sup_{\u \in (0, \infty )^n} \Big\{ | \X ^{\top } \U \X |^{d / 2} | \Y ^{\top } \U \Y - \Y ^{\top } \U \X ( \X ^{\top } \U \X )^{- 1} \X ^{\top } \U \Y |^{(n - p - d - 1) / 2 + c} \prod_{i = 1}^{n} e^{- b u_i / 2} \Big\} < \infty \text{,} \non 
\end{align}
we have 
\begin{align}
&1( u_1 \le \dots \le u_n ) p( \bBe , \bSit | \u ) p( \u | \bBe , \bSit ) \non \\
&\le M_3 {1 \over | \bSit |^{n / 2 + c} (1 + Q)^{2 [d \{ n / 2 + c - (d + 1) / 2 \} + e]}} \non \\
&\quad \times \frac{ \prod_{i = 1}^{n} ((1 + 1 / u_i )^{2 [d \{ n / 2 + c - (d + 1) / 2 \} + e] / (n + 1 - p - d)} {u_i}^{d / 2 + a - 1} e^{- b u_i / 2} ) }{ \int_{0}^{\infty } \big[ \al ^{d \{ n / 2 + c - (d + 1) / 2 \} + e - 1} e^{- f \al } / \prod_{i = 1}^{n} (1 + \al Q)^{d / 2 + a} \big] d\al } \non 
\end{align}
for some $M_3 > 0$, where $Q = \tr ( \bSit ^{- 1} ) + \tr ( \bBe \bSit ^{- 1} \bBe ^{\top } )$. 
Note that 
\begin{align}
\int_{0}^{\infty } {\al ^{d \{ n / 2 + c - (d + 1) / 2 \} + e - 1} e^{- f \al } \over \prod_{i = 1}^{n} (1 + \al Q)^{d / 2 + a}} d\al &= \int_{0}^{\infty } {\alt ^{d \{ n / 2 + c - (d + 1) / 2 \} + e - 1} e^{- f \alt / Q} \over Q^{d \{ n / 2 + c - (d + 1) / 2 \} + e} \prod_{i = 1}^{n} (1 + \alt )^{d / 2 + a}} d\alt \non \\
&\ge {1 \over Q^{d \{ n / 2 + c - (d + 1) / 2 \} + e}} \int_{0}^{\infty } {\alt ^{d \{ n / 2 + c - (d + 1) / 2 \} + e - 1} e^{- f \alt } \over \prod_{i = 1}^{n} (1 + \alt )^{d / 2 + a}} d\alt \non 
\end{align}
if $Q \ge 1$ and that 
\begin{align}
\int_{0}^{\infty } {\al ^{d \{ n / 2 + c - (d + 1) / 2 \} + e - 1} e^{- f \al } \over \prod_{i = 1}^{n} (1 + \al Q)^{d / 2 + a}} d\al &\ge \int_{0}^{\infty } {\al ^{d \{ n / 2 + c - (d + 1) / 2 \} + e - 1} e^{- f \al } \over \prod_{i = 1}^{n} (1 + \al )^{d / 2 + a}} d\al \non 
\end{align}
if $Q < 1$. 
Then 
\begin{align}
&1( u_1 \le \dots \le u_n ) p( \bBe , \bSit | \u ) p( \u | \bBe , \bSit ) \non \\
&\le M_4 {(1 + Q)^{d \{ n / 2 + c - (d + 1) / 2 \} + e} \over | \bSit |^{n / 2 + c} (1 + Q)^{2 [d \{ n / 2 + c - (d + 1) / 2 \} + e]}} \non \\
&\quad \times \prod_{i = 1}^{n} ((1 + 1 / u_i )^{2 [d \{ n / 2 + c - (d + 1) / 2 \} + e] / (n + 1 - p - d)} {u_i}^{d / 2 + a - 1} e^{- b u_i / 2} ) \non 
\end{align}
for some $M_4 > 0$ and 
\begin{align}
I' &\le M_5 \int_{\mathbb{R} ^{p \times d} \times \mathcal{S} ^{>} (d)} {1 \over | \bSit |^{n / 2 + c} (1 + Q)^{d \{ n / 2 + c - (d + 1) / 2 \} + e}} d( \bBe , \bSit ) \non 
\end{align}
for some $M_5 > 0$, which is finite by Lemma \ref{lem:t_beta}. 
\hfill$\Box$

\begin{lem}
\label{lem:denominator} 
There exists $M > 0$ such that 
\begin{align}
( \y _i - \bBe ^{\top } \x _i )^{\top } \bSit ^{- 1} ( \y _i - \bBe ^{\top } \x _i ) &\le M \{ \tr ( \bSit ^{- 1} ) + \tr ( \bBe \bSit ^{- 1} \bBe ^{\top } ) \} \non 
\end{align}
for all $i = 1, \dots , n$. 
\end{lem}

\noindent
{\bf Proof
.} \ \ Let $M > 0$ be such that $M \I ^{(d)} > \X ^{\top } \X + \Y ^{\top } \Y $. 
Then 
\begin{align}
( \y _i - \bBe ^{\top } \x _i )^{\top } \bSit ^{- 1} ( \y _i - \bBe ^{\top } \x _i ) &\le 2 ( {\y _i}^{\top } \bSit ^{- 1} \y _i + {\x _i}^{\top } \bBe \bSit ^{- 1} \bBe ^{\top } \x _i ) \non \\
&= 2 \{ \tr ( \bSit ^{- 1} \y _i {\y _i}^{\top } ) + \tr ( \bBe \bSit ^{- 1} \bBe ^{\top } \x _i {\x _i}^{\top } ) \} \non \\
&\le M \{ \tr ( \bSit ^{- 1} ) + \tr ( \bBe \bSit ^{- 1} \bBe ^{\top } ) \} \non 
\end{align}
for all $i = 1, \dots , n$. 
\hfill$\Box$

\begin{lem}
\label{lem:numerator} 
\hfill
\begin{itemize}
\item[{\rm{(i)}}]
There exists $M' > 0$ such that 
\begin{align}
&\sum_{i \in \{ i_1 , \dots , i_{p + d} \} } u_i ( \y _i - \bBe ^{\top } \x _i )^{\top } \bSit ^{- 1} ( \y _i - \bBe ^{\top } \x _i ) \ge {1 \over M'} ( \min_{i \in \{ i_1 , \dots , i_{p + d} \} } u_i ) \{ \tr ( \bSit ^{- 1} ) + \tr ( \bBe \bSit ^{- 1} \bBe ^{\top } ) \} \non 
\end{align}
for any $1 \le i_1 < \dots < i_{p + d} \le n$. 
\item[{\rm{(ii)}}]
There exists $M'' > 0$ such that 
\begin{align}
&{1 \over f + \sum_{i = 1}^{n} u_i ( \y _i - \bBe ^{\top } \x _i )^{\top } \bSit ^{- 1} ( \y _i - \bBe ^{\top } \x _i ) / 2} \le M'' {\prod_{i = 1}^{n + 1 - p - d} (1 + 1 / u_i )^{1 / (n + 1 - p - d)} \over 1 + \tr ( \bSit ^{- 1} ) + \tr ( \bBe \bSit ^{- 1} \bBe ^{\top } )} \non 
\end{align}
if $u_1 \le \dots \le u_n$. 
\end{itemize}
\end{lem}

\noindent
{\bf Proof
.} \ \ For part (i), we can assume that $u_1 = \dots = u_n = 1$. 
Fix $1 \le i_1 < \dots < i_{p + d} \le n$, let $\i = ( i_1 , \dots , i_{p + d} )$, and let $\X _{\i } = ( \x _{i_1} , \dots , \x _{i_{p + d}} )^{\top }$ and $\Y _{\i } = ( \y _{i_1} , \dots , \y _{i_{p + d}} )^{\top }$. 
Then, since 
\begin{align}
&\begin{pmatrix} {\X _{\i }}^{\top } \X _{\i } & {\X _{\i }}^{\top } \Y _{\i } \\ {\Y _{\i }}^{\top } \X _{\i } & {\Y _{\i }}^{\top } \Y _{\i } \end{pmatrix} = \begin{pmatrix} \X _{\i } & \Y _{\i } \end{pmatrix} ^{\top } \begin{pmatrix} \X _{\i } & \Y _{\i } \end{pmatrix} > \O ^{(p + d)} \text{,} \non 
\end{align}
we have ${\X _{\i }}^{\top } \X _{\i } > \O ^{(p)}$ and ${\Y _{\i }}^{\top } \Y _{\i } - {\Y _{\i }}^{\top } \X _{\i } ( {\X _{\i }}^{\top } \X _{\i } )^{- 1} {\X _{\i }}^{\top } \Y _{\i } > \O ^{(d)}$, which imples that ${\X _{\i }}^{\top } \X _{\i } > \I ^{(p)} / M_1$ and ${\Y _{\i }}^{\top } \Y _{\i } - {\Y _{\i }}^{\top } \X _{\i } ( {\X _{\i }}^{\top } \X _{\i } )^{- 1} {\X _{\i }}^{\top } \Y _{\i } > \I ^{(d)} / M_1$ for some $M_1 > 0$. 
Therefore, 
\begin{align}
\sum_{i \in \{ i_1 , \dots , i_{p + d} \} } ( \y _i - \bBe ^{\top } \x _i )^{\top } \bSit ^{- 1} ( \y _i - \bBe ^{\top } \x _i ) \non &= \tr \{ \bSit ^{- 1} ( \Y _{\i } - \X _{\i } \bBe )^{\top } ( \Y _{\i } - \X _{\i } \bBe ) \} \non \\
&= \tr [ \bSit ^{- 1} ( \bBe - \B _{\i } )^{\top } {\X _{\i }}^{\top } \X _{\i } ( \bBe - \B _{\i } ) \non \\
&\quad + \bSit ^{- 1} \{ {\Y _{\i }}^{\top } \Y _{\i } - {\Y _{\i }}^{\top } \X _{\i } ( {\X _{\i }}^{\top } \X _{\i } )^{- 1} {\X _{\i }}^{\top } \Y _{\i } \} ] \non \\
&\ge [ \tr ( \bSit ^{- 1} ) + \tr \{ ( \bBe - \B _{\i } ) \bSit ^{- 1} ( \bBe - \B _{\i } )^{\top } \} ] / M_1 \text{,} \label{lnumeratorp1} 
\end{align}
where $\B _{\i } = ( {\X _{\i }}^{\top } \X _{\i } )^{- 1} {\X _{\i }}^{\top } \Y _{\i }$. 
Furthermore, since 
\begin{align}
\v ^{\top } \bBe \bSit ^{- 1} \bBe ^{\top } \v &\le 2 \{ \v ^{\top } ( \bBe - \B _{\i } ) \bSit ^{- 1} ( \bBe - \B _{\i } )^{\top } \v + \v ^{\top } \B _{\i } \bSit ^{- 1} {\B _{\i }}^{\top } \v \} \non 
\end{align}
for all $\v \in \mathbb{R} ^p$, 
\begin{align}
\tr ( \bSit ^{- 1} ) + \tr ( \bBe \bSit ^{- 1} \bBe ^{\top } ) &\le \tr ( \bSit ^{- 1} ) + 2 \tr \{ ( \bBe - \B _{\i } ) \bSit ^{- 1} ( \bBe - \B _{\i } )^{\top } + \B _{\i } \bSit ^{- 1} {\B _{\i }}^{\top } \} \non \\
&\le M_2 [ \tr ( \bSit ^{- 1} ) + \tr \{ ( \bBe - \B _{\i } ) \bSit ^{- 1} ( \bBe - \B _{\i } )^{\top } \} ] \label{lnumeratorp2} 
\end{align}
for some $M_2 > 0$. 
The desired result follows from (\ref{lnumeratorp1}) and (\ref{lnumeratorp2}). 

For part (ii), suppose that $u_1 \le \dots \le u_n$. 
Then 
\begin{align}
&{1 \over f + \sum_{i = 1}^{n} u_i ( \y _i - \bBe ^{\top } \x _i )^{\top } \bSit ^{- 1} ( \y _i - \bBe ^{\top } \x _i ) / 2} \non \\
&\le M_3 {1 \over 1 + \sum_{i = 1}^{n} u_i ( \y _i - \bBe ^{\top } \x _i )^{\top } \bSit ^{- 1} ( \y _i - \bBe ^{\top } \x _i )} \non \\
&\le M_3 \Big\{ \prod_{i = 1}^{n + 1 - p - d} {1 \over 1 + \sum_{i' = i}^{i + p + d - 1} u_{i'} ( \y _{i'} - \bBe ^{\top } \x _{i'} )^{\top } \bSit ^{- 1} ( \y _{i'} - \bBe ^{\top } \x _{i'} )} \Big\} ^{1 / (n + 1 - p - d)} \non \\
&\le M_3 \Big[ \prod_{i = 1}^{n + 1 - p - d} {1 \over 1 + u_i \{ \tr ( \bSit ^{- 1} ) + \tr ( \bBe \bSit ^{- 1} \bBe ^{\top } ) \} / M'} \Big] ^{1 / (n + 1 - p - d)} \non \\
&\le M_4 \Big\{ \prod_{i = 1}^{n + 1 - p - d} {\max \{ 1, 1 / u_i \} \over 1 + \tr ( \bSit ^{- 1} ) + \tr ( \bBe \bSit ^{- 1} \bBe ^{\top } )} \Big\} ^{1 / (n + 1 - p - d)} \non 
\end{align}
for some $M_3 , M_4 > 0$. 
\hfill$\Box$

\begin{lem}
\label{lem:t_beta} 
Suppose that $n - p > 2 (d - c)$ and that $e > d (d - 1) / 2$. 
Then 
\begin{align}
&\int_{\mathbb{R} ^{p \times d} \times \mathcal{S} ^{>} (d)} {1 \over | \bSit |^{n / 2 + c} \{ 1 + \tr ( \bSit ^{- 1} ) + \tr ( \bBe \bSit ^{- 1} \bBe ^{\top } ) \} ^{d \{ n / 2 + c - (d + 1) / 2 \} + e}} d( \bBe , \bSit ) < \infty \text{.} \non 
\end{align}
\end{lem}

\noindent
{\bf Proof
.} \ \ Fix $\ep > 0$. 
Then 
\begin{align}
&{1 \over | \bSit |^{n / 2 + c} \{ 1 + \tr ( \bSit ^{- 1} ) + \tr ( \bBe \bSit ^{- 1} \bBe ^{\top } ) \} ^{d \{ n / 2 + c - (d + 1) / 2 \} + e}} \non \\
&\le {1 \over | \bSit |^{(n - p) / 2 + c}} {1 \over \{ 1 + \tr ( \bSit ^{- 1} ) \} ^{d \{ n / 2 + c - (d + 1) / 2 - p (1 + \ep ) / 2 \} + e}} {1 \over | \bSit |^{p / 2}} {1 \over \{ 1 + \tr ( \bBe \bSit ^{- 1} \bBe ^{\top } ) \} ^{p d (1 + \ep ) / 2}} \text{,} \non 
\end{align}
where 
\begin{align}
&\{ 1 + \tr ( \bSit ^{- 1} ) \} ^d \ge | \I ^{(d)} + \bSit ^{- 1} | \text{.} \non 
\end{align}
Suppose that $n / 2 + c - (d + 1) / 2 - p (1 + \ep ) / 2 + e / d \ge 0$. 
Then, by making the change of variables $\bBe = \widetilde{\bBe } ( \bSit ^{1 / 2} )^{\top }$, 
\begin{align}
&\int_{\mathbb{R} ^{p \times d} \times \mathcal{S} ^{>} (d)} {1 \over | \bSit |^{n / 2 + c} \{ 1 + \tr ( \bSit ^{- 1} ) + \tr ( \bBe \bSit ^{- 1} \bBe ^{\top } ) \} ^{d \{ n / 2 + c - (d + 1) / 2 \} + e}} d( \bBe , \bSit ) \non \\
&\le \int_{\mathbb{R} ^{p \times d} \times \mathcal{S} ^{>} (d)} {1 / | \bSit |^{(n - p) / 2 + c} \over | \I ^{(d)} + \bSit ^{- 1} |^{n / 2 + c - (d + 1) / 2 - p (1 + \ep ) / 2 + e / d}} {1 \over \{ 1 + \tr ( \widetilde{\bBe } \widetilde{\bBe } ^{\top } ) \} ^{p d (1 + \ep ) / 2}} d( \widetilde{\bBe } , \bSit ) \non \\
&\le \Big\{ \prod_{k = 1}^{p} \prod_{j = 1}^{d} \int_{\mathbb{R} ^{p \times d}} {1 \over (1 + \bet _{k, j}^2 )^{(1 + \ep ) / 2}} d{\bet _{k, j}} \Big\} \int_{\mathcal{S} ^{>} (d)} {| \bSit |^{n / 2 + c - (d + 1) / 2 - p (1 + \ep ) / 2 + e / d - \{ (n - p) / 2 + c \} } \over | \I ^{(d)} + \bSit |^{n / 2 + c - (d + 1) / 2 - p (1 + \ep ) / 2 + e / d}} d\bSit \non 
\end{align}
where $\bet _{k, j} = ( \widetilde{\bBe } )_{k, j}$ for $j = 1, \dots , d$ for $k = 1, \dots , p$. 
The right-hand side is finite if $n / 2 + c - (d + 1) / 2 - p (1 + \ep ) / 2 + e / d - \{ (n - p) / 2 + c \} > - 1$ and $(n - p) / 2 + c > d$; see, for example, (5.2.2) of \cite{gn2000} for the matrix variate beta type II distribution. 
It follows from the assumptions of the lemma that if $\ep > 0$ is chosen appropriately, the conditions above are satisfied. 
\hfill$\Box$

\section{A Simulation Study}
\label{sec:sim} 
As an example, we consider the problem of estimating the spectral gap of the Markov operator associated with Algorithm \ref{algo:trace_class}, following the approach proposed by \cite{qhk2019}. 
We assume that $n \ge p + 1$ and that $1 / 2 + a > n / (n - p)$ and suppress the dependence on $\y $.

\subsection{Preliminaries}
\label{subsec:preliminaries} 
The statements in this section follow from \cite{qhk2019}. 
The total variation distance between the distribution of the $m$th element of the Markov chain and the stationary distribution is bounded above by some constant times $(1 - \de _{*} )^m$ for all $m \in \mathbb{N}$, where $\de _{*} \in (0, 1]$ is the spectral gap of the Markov operator. 
There exist $\infty > s_1 \ge s_2 \ge \dots \ge 1$ such that $1 - \de _{*} \le ( s_{m + 1} - 1)^{1 / (m + 1)} \le ( s_m - 1)^{1 / m}$ for all $m \in \mathbb{N}$ and such that $\lim_{m \to \infty } ( s_m - 1)^{1 / m} = 1 - \de _{*}$. 
For any positive density $\psi ^{*}$ on $\mathbb{R} ^p$ and any $m \in \mathbb{N}$, we have that 
\begin{itemize}
\item
if $\bbe ^{*} \sim ( \text{the density $\psi ^{*}$} )$, 
\item
if $\bbe ( m' ) \sim ( \text{the transition kernel from $\bbe ( m' - 1)$} )$, $m' = 2, \dots , m$, with $\bbe (1) = \bbe ^{*}$, 
\item
and if $\u ^{*} \sim ( \text{the conditional distribution of $\u $ given $\bbe = \bbe (m)$} )$, 
\end{itemize}
then 
\begin{align}
s_m &= E \Big[ {g( \bbe ^{*} , \u ^{*} ) \over \psi ^{*} ( \bbe ^{*} )} \Big] \label{eq:unbiased_estimator} 
\end{align}
for all $m \in \mathbb{N}$. 
Here, 
\begin{align}
g( \bbe , \u ) &= {\Ga (n / 2) / \{ \pi ^{p / 2} \Ga ((n - p) / 2) \} \over \{ ( \y - \X \bbe )^{\top } \U ( \y - \X \bbe ) \} ^{n / 2}} | \X ^{\top } \U \X |^{1 / 2} \{ \y ^{\top } \U \y - \y ^{\top } \U \X ( \X ^{\top } \U \X )^{- 1} \X ^{\top } \U \y \} ^{(n - p) / 2} \text{,} \non 
\end{align}
where $\U = \bdiag \u $, for $( \bbe , \u ) \in \mathbb{R} ^p \times \mathbb{R} ^n$ (see (\ref{eq:p_be_u})). 
Thus, we can construct a Monte Carlo estimator based on iid samples generated in this way for each of $s_1 , s_2 , \dotsc $. 
Moreover, we can use confidence intervals for $s_1 , s_2 , \dotsc $ to construct those for $1 - \de _{*}$.

\subsection{The density $\psi ^{*}$}
Based on the discussion in Section 6 of \cite{qhk2019}, we try to construct a density which resembles $p( \bbe )$ in defining the density $\psi ^{*}$. 
In particular, we substitute into $p( \bbe | \u ) = g( \bbe , \u )$ the estimator $\ubh $ of $\u $ given by 
\begin{align}
\ubh &= E[ \u | ( \bbe , \si ^2 ) = ( \bbeh , \sih ^2 ) ] = \Big( {1 / 2 + a \over ( y_i - {\x _i}^{\top } \bbeh )^2 / (2 \sih ^2 ) + b} \Big) _{i = 1}^{n} \text{,} \non 
\end{align}
where $( \bbeh , \sih ^2 )$ is the posterior mode of $( \bbe , \si ^2 )$. 
The corresponding unbiased estimator of $s_m$ can have a finite variance for all $m \in \mathbb{N}$.

\begin{prp}
\label{prp:finite_variance} 
Fix $m \in \mathbb{N}$ and let $( \bbe ^{*} , \u ^{*} )$ be as in Section \ref{subsec:preliminaries} with $\psi $ defined as above. 
Then 
\begin{align}
Var \Big( {g( \bbe ^{*} , \u ^{*} ) \over g( \bbe ^{*} , \ubh )} \Big) &< \infty \non 
\end{align}
provided that $n \ge p + 1$ and that $1 / 2 + a > 2 n / (n - p)$. 
\end{prp}

\bigskip

\noindent
{\bf Proof
.} \ \ By Corollary 5 of \cite{qhk2019}, it is sufficient to check that 
\begin{align}
\int_{\mathbb{R} ^p \times \mathbb{R} ^n} {\{ p( \bbe | \u ) \} ^3 p( \u | \bbe ) \over \{ g( \bbe , \ubh ) \} ^2} d( \bbe , \u ) &< \infty \text{.} \non 
\end{align}
As in the proof of Theorem \ref{prp:trace_class}, we have 
\begin{align}
&{\{ p( \bbe | \u ) \} ^3 p( \u | \bbe ) \over \{ g( \bbe , \ubh ) \} ^2} \non \\
&\le M_1 \Big[ {1 \over \{ ( \y - \X \bbe )^{\top } \U ( \y - \X \bbe ) \} ^{n / 2}} | \X ^{\top } \U \X |^{1 / 2} \{ \y ^{\top } \U \y - \y ^{\top } \U \X ( \X ^{\top } \U \X )^{- 1} \X ^{\top } \U \y \} ^{(n - p) / 2} \Big] ^3 \non \\
&\quad \times \Big( {\prod_{i = 1}^{n} ( {u_i}^{1 / 2 + a - 1} e^{- b u_i} ) \over \big\{ \sum_{i = 1}^{n} u_i ( y_i - {\x _i}^{\top } \bbe )^2 \big\} ^{n / 2}} / \int_{0}^{\infty } ( \sis )^{n a - 1} \Big[ \prod_{i = 1}^{n} {1 \over \{ ( y_i - {\x _i}^{\top } \bbe )^2 / (2 b) + \sis \} ^{1 / 2 + a}} \Big] d\sis \Big) \non \\
&\quad / \Big[ {1 \over \{ ( \y - \X \bbe )^{\top } \widehat{\U } ( \y - \X \bbe ) \} ^{n / 2}} \Big] ^2 \non \\
&\le M_2 {1 \over \{ ( \y - \X \bbe )^{\top } \U ( \y - \X \bbe ) \} ^{3 n / 2}} {\prod_{i = 1}^{n} ( {u_i}^{1 / 2 + a - 1} e^{- b u_i / 2} ) \over \big\{ \sum_{i = 1}^{n} u_i ( y_i - {\x _i}^{\top } \bbe )^2 \big\} ^{n / 2}} / \Big\{ {1 \over (1 + \| \bbe \| ^2 )^{n / 2}} {1 \over ( \tr \widehat{\U } )^n (1 + \| \bbe \| ^2 )^n} \Big\} \non \\
&\le M_3 {\prod_{i = 1}^{n} ( {u_i}^{1 / 2 + a - 1} e^{- b u_i / 2} ) \over \big\{ \sum_{i = 1}^{n} u_i ( y_i - {\x _i}^{\top } \bbe )^2 \big\} ^{2 n}} (1 + \| \bbe \| ^2 )^{3 n / 2} \non 
\end{align}
for some $M_1 , M_2 , M_3 > 0$, where $\widehat{\U } = \bdiag \ubh $. 
Now, it is sufficient to show that $I' < \infty $, where 
\begin{align}
I' &= \int_{\mathbb{R} ^p \times \mathbb{R} ^n} 1( u_1 \le \dots \le u_n ) {\prod_{i = 1}^{n} ( {u_i}^{1 / 2 + a - 1} e^{- b u_i / 2} ) \over \big\{ \sum_{i = 1}^{n} u_i ( y_i - {\x _i}^{\top } \bbe )^2 \big\} ^{2 n}} (1 + \| \bbe \| ^2 )^{3 n / 2} d( \bbe , \u ) \text{.} \non 
\end{align}
As in the proof of Theorem \ref{prp:trace_class}, 
\begin{align}
I' &\le M_4 \int_{\mathbb{R} ^p \times \mathbb{R} ^n} {1 \over (1 + \| \bbe \| ^2 )^{n / 2}} \Big( \prod_{i = 1}^{n} [ \{ 1 + 1 / {u_i}^{n / (n - p)} \} ^2 {u_i}^{1 / 2 + a - 1} e^{- b u_i / 2} ] \Big) d( \bbe , \u ) \text{,} \non 
\end{align}
the right-hand side of which is finite by assumption. 
This completes the proof. 
\hfill$\Box$

\subsection{The posterior mode $( \bbeh , \sih ^2 )$}
Based on the relation $p( \bbe , \si ^2 ) = \int_{(0, \infty )^n} p( \bbe , \si ^2 , \u ) d\u $, we use the EM algorithm to find the posterior mode $( \bbeh , \sih ^2 )$. 
By Section \ref{subsec:conditional_distributions}, for all $( \bbe _{\rm{old}} , \si _{\rm{old}}^{2} ) , ( \bbe _{\rm{new}} , \si _{\rm{new}}^{2} ) \in \mathbb{R} ^p \times (0, \infty )$ we have 
\begin{align}
&E[ \log p( \bbe _{\rm{new}} , \si _{\rm{new}}^{2} , \u ) | ( \bbe , \si ^2 ) = ( \bbe _{\rm{old}} , \si _{\rm{old}}^{2} ) ] \non \\
&= C_1 - \Big( 1 + {n \over 2} \Big) \log \si _{\rm{new}}^{2} \non \\
&\quad + \sum_{i = 1}^{n} \Big( \Big( {1 \over 2} + a - 1 \Big) \Big[ \psi \Big( {1 \over 2} + a \Big) - \log \Big\{ {( y_i - {\x _i}^{\top } \bbe _{\rm{old}} )^2 \over 2 \si _{\rm{old}}^{2}} + b \Big\} \non \\
&\quad - {1 / 2 + a \over ( y_i - {\x _i}^{\top } \bbe _{\rm{old}} )^2 / (2 \si _{\rm{old}}^{2} ) + b} \Big\{ b + {( y_i - {\x _i}^{\top } \bbe _{\rm{new}} )^2 \over 2 \si _{\rm{new}}^{2}} \Big\} \Big\} \Big] \Big) \non \\
&= C_2 - \Big( 1 + {n \over 2} \Big) \log \si _{\rm{new}}^{2} - \sum_{i = 1}^{n} \wt _i {( y_i - {\x _i}^{\top } \bbe _{\rm{new}} )^2 \over 2 \si _{\rm{new}}^{2}} \non \\
&= C_2 - \Big( 1 + {n \over 2} \Big) \log \si _{\rm{new}}^{2} - {\y ^{\top } \{ \Wbt - \Wbt \X ( \X ^{\top } \Wbt \X )^{- 1} \X ^{\top } \Wbt \} \y \over 2 \si _{\rm{new}}^{2}} \non \\
&\quad - {\{ \bbe _{\rm{new}} - ( \X ^{\top } \Wbt \X )^{- 1} \X ^{\top } \Wbt \y \} ^{\top } ( \X ^{\top } \Wbt \X ) \{ \bbe _{\rm{new}} - ( \X ^{\top } \Wbt \X )^{- 1} \X ^{\top } \Wbt \y \} \over 2 \si _{\rm{new}}^{2}} \non 
\end{align}
for some $C_1 , C_2 \in \mathbb{R}$ independent of $( \bbe _{\rm{new}} , \si _{\rm{new}}^{2} )$, where the first equality follows from the fact that $E[ \log {\rm{Ga}} ( a' , b' ) ] = \psi ( a' ) - \log b'$ for all $a' , b' > 0$ and where $\wt _i = (1 / 2 + a) / \{ ( y_i - {\x _i}^{\top } \bbe _{\rm{old}} )^2 / (2 \si _{\rm{old}}^{2} ) + b \} $ for $i = 1, \dots , n$ and $\Wbt = \bdiag ( \wt _i )_{i = 1}^{n}$. 
Thus, the EM algorithm is as follows. 

\begin{algo}
\label{algo:EM} 
Given a current value $( \bbe _{\rm{old}} , \si _{\rm{old}}^{2} ) \in \mathbb{R} ^p \times (0, \infty )$, 
\begin{itemize}
\item
let $\Wbt = \bdiag ((1 / 2 + a) / \{ ( y_i - {\x _i}^{\top } \bbe _{\rm{old}} )^2 / (2 \si _{\rm{old}}^{2} ) + b \} )_{i = 1}^{n}$, 
\item
let $\bbe _{\rm{new}} = ( \X ^{\top } \Wbt \X )^{- 1} \X ^{\top } \Wbt \y $, and 
\item
let $\si _{\rm{new}}^{2} = [ \y ^{\top } \{ \Wbt - \Wbt \X ( \X ^{\top } \Wbt \X )^{- 1} \X ^{\top } \Wbt \} \y  / 2] / (1 + n / 2)$. 
\end{itemize}
\end{algo}

\subsection{Results}
Here, we illustrate the procedure in a simple case. 
We set $n = 20$ and $p = 3$ and generate a single pair of $\y = ( y_i )_{i = 1}^{n}$ and $\X = (( x_{i, k} )_{k = 1, \dots , p} )_{i = 1}^{n}$ in the following way: 
\begin{align}
&x_{1, 1} = \dots = x_{n, 1} = 1 \text{,} \non \\
&x_{i, k} \sim {\rm{N}} (0, 1) \text{,} \quad k = 2, 3 \text{,} \quad i = 1, \dots , n \text{,} \non \\
&\mathring{u} _i \sim {\rm{Ga}} (3, 1) \text{,} \quad i = 1, \dots , n \text{,} \non \\
&y_i \sim {\rm{N}}_p ( {\x _i}^{\top } (1, - 2, 3)^{\top } , 1 / \mathring{u} _i ) \text{,} \quad i = 1, \dots , n \text{.} \non 
\end{align}
We consider the Markov operator associated with Algorithm \ref{algo:trace_class} when $a = 3$ and $b = 1$. 
We generate $N$ Monte Carlo samples of $g( \bbe ^{*} , \u ^{*} ) / \psi ^{*} ( \bbe ^{*} )$ to obtain estimates of $s_m$ and the upper bound $( s_m - 1)^{1 / m}$ for $m = 1, \dots , 5$ for $N = {10}^5 , {10}^6$. 

Estimates of the upper bounds on $1 - \de _{*}$ are shown in Figure \ref{fig:upper}. 
The convergence rate $1 - \de _{*}$ will be in the interval $[0, 0.4]$, for example, if the fifth upper bound $( s_5 - 1)^{1 / 5}$ is equal to $0.4$. 
If the sequence has converged approximately, then $1 - \de _{*} \approx 0.4$. 

\begin{figure}[H]
\centering
\includegraphics[width = 1.0\linewidth]{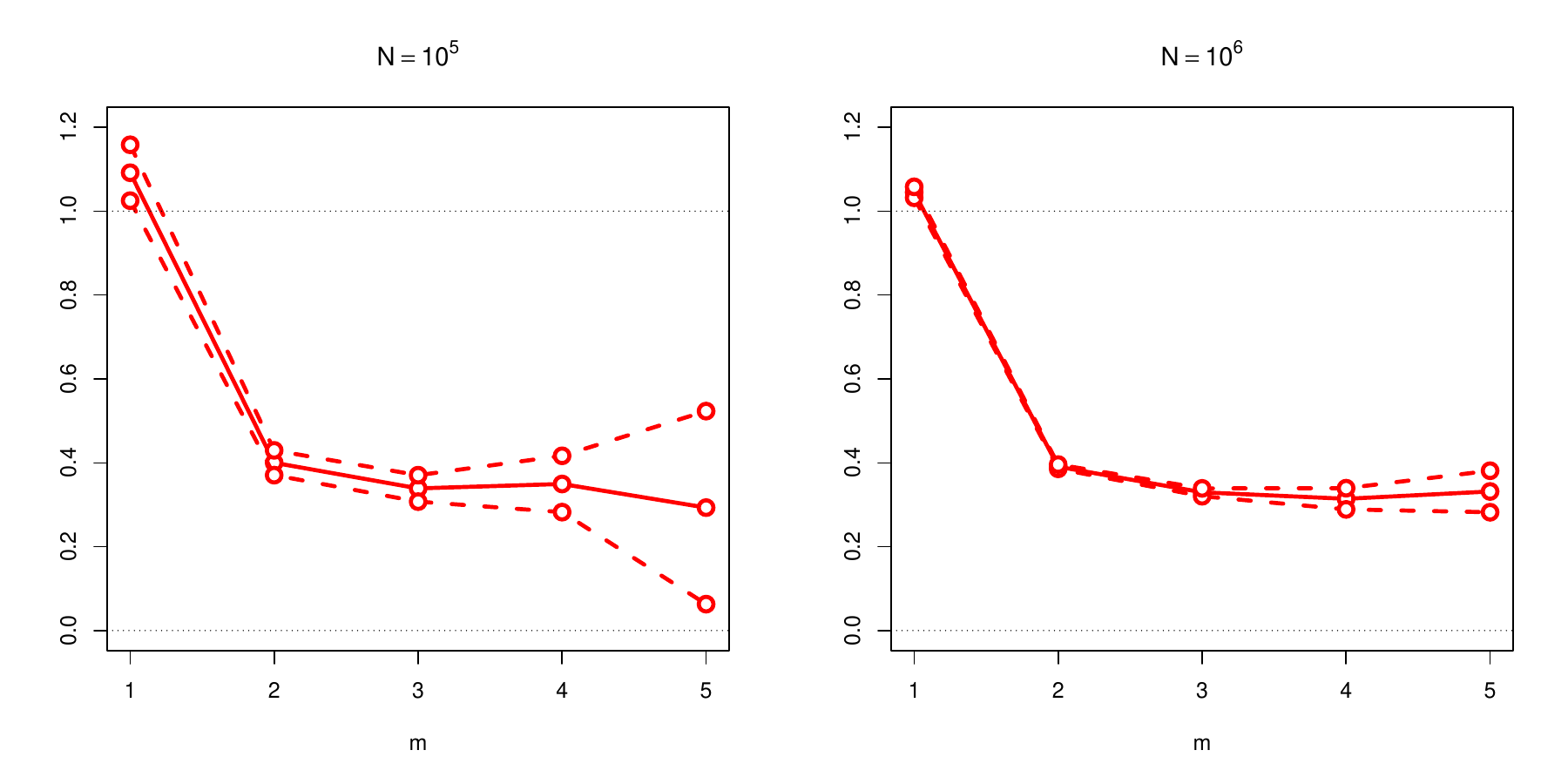}
\caption{Estimates of the upper bounds $( s_m - 1)^{1 / m}$, $m = 1, \dots , 5$, when $N = {10}^5$ (left) and when $N = {10}^6$ (right). 
The solid line segments correspond to point estimates. 
The dashed line segments correspond to approximate $95 \% $ conficence intervals based on the delta method. 
}
\label{fig:upper} 
\end{figure}%

\subsection{Implications for the mixing time}
\label{subsec:mixing_time} 
Let $\mu $ and $q$ be the initial distribution and density for $\bbe $ and fix $m \in \mathbb{N}$. 
By Chapter 19 of \cite{cvjbgm2026}, the total variation distance between the distribution of the $m$th element of the Markov chain, denoted by $\mu _m$, and the stationary distribution, denoted by $\varpi $, is 
\begin{align}
{\rm{TV}} ( \mu _m , \varpi ) &\le %
{1 \over 2} \| \mu - \varpi \| _2 (1 - \de _{*} )^m \text{,} \non 
\end{align}
where 
\begin{align}
\| \mu - \varpi \| _2 %
&= \sqrt{\int_{\mathbb{R} ^p} \Big\{ {q( \bbe ) \over p( \bbe )} - 1 \Big\} ^2 p( \bbe ) d\bbe } \non 
\end{align}
is called the $L^2$ distance between $\mu $ and $\varpi $. 
Since 
\begin{align}
{\| \mu - \varpi \| _2}^2 &\le 2 \int_{\mathbb{R} ^p} \Big[ \Big\{ {q( \bbe ) \over p( \bbe )} \Big\} ^2 + 1 \Big] p( \bbe ) d\bbe = 2 \Big[ \int_{\mathbb{R} ^p} {\{ q( \bbe ) \} ^2 \over p( \bbe )} d\bbe + 1 \Big] \text{,} \non 
\end{align}
it follows that 
\begin{align}
{\rm{TV}} ( \mu _m , \varpi ) &\le C (1 - \de _{*} )^m \text{,} \non 
\end{align}
where 
\begin{align}
C &= {1 \over \sqrt{2}} \sqrt{\int_{\mathbb{R} ^p} {\{ q( \bbe ) \} ^2 \over p( \bbe )} d\bbe + 1} \text{.} \non 
\end{align}
Thus, for any $\ep > 0$, we have that 
\begin{align}
{\rm{TV}} ( \mu _m , \varpi ) &< \epsilon \non 
\end{align}
if 
\begin{align}
1 - \de _{*} &< ( \ep / C)^{1 / m} \non 
\end{align}
and an upper bound on the mixing time is obtained in principle. 

In practice, we have to compute an upper bound on the constant $C$. 
By Section \ref{subsec:conditional_distributions}, 
\begin{align}
p( \bbe ) &= \int_{0}^{\infty } \Big\{ \int_{\mathbb{R} ^n} p( \bbe , \sis , \u ) d\u \Big\} d\sis \non \\
&\propto \int_{0}^{\infty } {1 \over ( \sis )^{n / 2 + 1}} \Big[ \prod_{i = 1}^{n} {\Ga (1 / 2 + a) \over \{ ( y_i - {\x _i}^{\top } \bbe )^2 / (2 \sis ) + b \} ^{1 / 2 + a}} \Big] d\sis \non \\
&\propto \int_{0}^{\infty } {1 \over \sis } \Big[ \prod_{i = 1}^{n} {( \sis )^a \over \{ ( y_i - {\x _i}^{\top } \bbe )^2 / 2 + b \sis \} ^{1 / 2 + a}} \Big] d\sis \text{.} \non 
\end{align}
Note that 
\begin{align}
\prod_{i = 1}^{n} {1 \over ( y_i - {\x _i}^{\top } \bbe )^2 / 2 + b \sis } &\ge \prod_{i = 1}^{n} {1 \over {y_i}^2 + ( {\x _i}^{\top } \bbe )^2 + b \sis } \non \\
&\ge \Big[ n / \sum_{i = 1}^{n} \{ {y_i}^2 + ( {\x _i}^{\top } \bbe )^2 + b \sis \} \Big] ^n = \Big( {1 \over \| \y \| ^2 / n + \bbe ^{\top } \X ^{\top } \X \bbe / n + b \sis } \Big) ^n \text{.} \non 
\end{align}
Then 
\begin{align}
p( \bbe ) &\ge {1 \over I} \int_{0}^{\infty } {( \sis )^{n a - 1} \over ( \| \y \| ^2 / n + \bbe ^{\top } \X ^{\top } \X \bbe / n + b \sis )^{n / 2 + n a}} d\sis \non \\
&= {1 \over I} \int_{0}^{\infty } {1 / b^{n a} \over ( \| \y \| ^2 / n + \bbe ^{\top } \X ^{\top } \X \bbe / n)^{n / 2}} {( \sis )^{n a - 1} \over (1 + \sis )^{n / 2 + n a}} d\sis = {1 \over I} {B(n a ,n / 2) / b^{n a} \over ( \| \y \| ^2 / n + \bbe ^{\top } \X ^{\top } \X \bbe / n)^{n / 2}} \text{,} \non 
\end{align}
where 
\begin{align}
I &= \int_{\mathbb{R} ^p \times (0, \infty )} {1 \over \sis } \Big[ \prod_{i = 1}^{n} {( \sis )^a \over \{ ( y_i - {\x _i}^{\top } \bbe )^2 / 2 + b \sis \} ^{1 / 2 + a}} \Big] d( \bbe , \sis ) \text{.} \non 
\end{align}
Therefore, 
\begin{align}
C &\le {1 \over \sqrt{2}} \sqrt{1 + {b^{n a} I J \over B(n a ,n / 2)}} \text{,} \non 
\end{align}
where 
\begin{align}
J &= \int_{\mathbb{R} ^p} ( \| \y \| ^2 / n + \bbe ^{\top } \X ^{\top } \X \bbe / n)^{n / 2} \{ q( \bbe ) \} ^2 d\bbe \text{.} \non 
\end{align}
The integrals $I$ and $J$ could be calculated using Monte Carlo integration. 
Additionally, closed-form bounds for these integrals can in principle be obtained, for example, when the initial density $q$ is multivariate normal; see Section \ref{sec:closed-form-bounds} of the Supplementary Material.

\section*{Acknowledgments}
Research of the author was supported in part by JSPS KAKENHI Grant Number JP25K21163 from Japan Society for the Promotion of Science.

\newpage
\setcounter{page}{1}
\setcounter{equation}{0}
\renewcommand{\theequation}{S\arabic{equation}}
\setcounter{section}{0}
\renewcommand{\thesection}{S\arabic{section}}
\setcounter{table}{0}
\renewcommand{\thetable}{S\arabic{table}}
\setcounter{figure}{0}
\renewcommand{\thefigure}{S\arabic{figure}}

\begin{center}
{\LARGE\bf Supplementary Materials}
\end{center}

\bigskip

\section{Proof of Theorem \ref{prp:proper}}
\label{sec:proof_prp_proper} 
Here, we prove Theorem \ref{prp:proper}. 

\bigskip

\noindent
{\bf Proof of Theorem \ref{prp:proper}.} \ \ Let $I = \int p( \bBe , \bSi | \u , \y ) p( \u | \bBe , \bSi , \y ) d( \bBe , \bSi , \u )$. 
We want to show that $I < \infty $. 
By Algorithm \ref{algo:DA}, 
\begin{align}
p( \u | \bBe , \bSi , \y ) %
&= \prod_{i = 1}^{n} {{u_i}^{d / 2 + a - 1} \exp [- u_i \{ ( \y _i - \bBe ^{\top } \x _i )^{\top } \bSi ^{- 1} ( \y _i - \bBe ^{\top } \x _i ) / 2 + b \} ] \over \Ga (d / 2 + a) / \{ ( \y _i - \bBe ^{\top } \x _i )^{\top } \bSi ^{- 1} ( \y _i - \bBe ^{\top } \x _i ) / 2 + b \} ^{d / 2 + a}} \non \\
&= C_1 \Big[ \prod_{i = 1}^{n} \{ ( \y _i - \bBe ^{\top } \x _i )^{\top } \bSi ^{- 1} ( \y _i - \bBe ^{\top } \x _i ) / 2 + b \} ^{d / 2 + a} \Big] \Big[ \prod_{i = 1}^{n} \{{u_i}^{d / 2 + a - 1} \exp (- b u_i ) \} \Big] \non \\
&\quad \times \exp \Big[ - {1 \over 2} \tr \{ \U ( \Y - \X \bBe ) \bSi ^{- 1} ( \Y - \X \bBe )^{\top } \} \Big] \non 
\end{align}
for some $C_1 > 0$, where 
\begin{align}
\tr \{ \U ( \Y - \X \bBe ) \bSi ^{- 1} ( \Y - \X \bBe )^{\top } \} %
&= \tr \{ \bSi ^{- 1} ( \bBe ^{\top } \X ^{\top } \U \X \bBe + \Y ^{\top } \U \Y - 2 \Y ^{\top } \U \X \bBe ) \} \text{.} \non 
\end{align}
Meanwhile, 
\begin{align}
p( \bBe , \bSi | \u , \y ) %
&= C_2 {| \bPsi |^{(n + \nu ) / 2} \over | \bSi |^{(n + \nu + d + 1) / 2}} \etr \Big( - {1 \over 2} \bPsi \bSi ^{- 1} \Big) {1 \over | \bOm |^{d / 2}} {1 \over | \bSi |^{p / 2}} \etr \Big\{ - {1 \over 2} \bOm ^{- 1} ( \bBe - \bGa ) \bSi ^{- 1} ( \bBe - \bGa )^{\top } \Big\} \non \\
&= {C_2 \over | \bOm |^{d / 2}} {| \bPsi |^{(n + \nu ) / 2} \over | \bSi |^{(n + \nu + d + 1) / 2}} {1 \over | \bSi |^{p / 2}} \etr \Big[ - {1 \over 2} \bSi ^{- 1} \{ \bPsi + ( \bBe - \bGa )^{\top } \bOm ^{- 1} ( \bBe - \bGa ) \} \Big] \non 
\end{align}
for some $C_2 > 0$, where 
\begin{align}
&\bPsi + ( \bBe - \bGa )^{\top } \bOm ^{- 1} ( \bBe - \bGa ) = \bPsi + \bBe ^{\top } \bOm ^{- 1} \bBe + \bGa ^{\top } \bOm ^{- 1} \bGa - 2 \bGa ^{\top } \bOm ^{- 1} \bBe \text{.} \non 
\end{align}
Therefore, 
\begin{align}
&p( \bBe , \bSi | \u , \y ) p( \u | \bBe , \bSi , \y ) \non \\
&= {C_2 \over | \bOm |^{d / 2}} {| \bPsi |^{(n + \nu ) / 2} \over | \bSi |^{(n + \nu + d + 1) / 2}} {1 \over | \bSi |^{p / 2}} \etr \Big\{ - {1 \over 2} \bSi ^{- 1} ( \bPsi + \bBe ^{\top } \bOm ^{- 1} \bBe + \bGa ^{\top } \bOm ^{- 1} \bGa - 2 \bGa ^{\top } \bOm ^{- 1} \bBe ) \Big\} \non \\
&\quad \times C_1 \Big[ \prod_{i = 1}^{n} \{ ( \y _i - \bBe ^{\top } \x _i )^{\top } \bSi ^{- 1} ( \y _i - \bBe ^{\top } \x _i ) / 2 + b \} ^{d / 2 + a} \Big] \Big[ \prod_{i = 1}^{n} \{ {u_i}^{d / 2 + a - 1} \exp (- b u_i ) \} \Big] \non \\
&\quad \times \etr \Big\{ - {1 \over 2} \bSi ^{- 1} ( \bBe ^{\top } \X ^{\top } \U \X \bBe + \Y ^{\top } \U \Y - 2 \Y ^{\top } \U \X \bBe ) \Big\} \non \\
&= {C_2 \over | \bOm |^{d / 2}} {| \bPsi |^{(n + \nu ) / 2} \over | \bSi |^{(n + \nu + d + 1) / 2}} {1 \over | \bSi |^{p / 2}} \etr \Big\{ - {1 \over 2} \bSi ^{- 1} ( \bPsi + \bGa ^{\top } \bOm ^{- 1} \bGa ) \Big\} \non \\
&\quad \times C_1 \Big[ \prod_{i = 1}^{n} \{ ( \y _i - \bBe ^{\top } \x _i )^{\top } \bSi ^{- 1} ( \y _i - \bBe ^{\top } \x _i ) / 2 + b \} ^{d / 2 + a} \Big] \Big[ \prod_{i = 1}^{n} \{ {u_i}^{d / 2 + a - 1} \exp (- b u_i ) \} \Big] \non \\
&\quad \times \etr \Big[ - {1 \over 2} \bSi ^{- 1} \{ \bBe ^{\top } ( \X ^{\top } \U \X + \bOm ^{- 1} ) \bBe + \Y ^{\top } \U \Y - 2 ( \Y ^{\top } \U \X + \bGa ^{\top } \bOm ^{- 1} ) \bBe \} \Big] \text{,} \non 
\end{align}
where 
\begin{align}
&\bBe ^{\top } ( \X ^{\top } \U \X + \bOm ^{- 1} ) \bBe + \Y ^{\top } \U \Y - 2 ( \Y ^{\top } \U \X + \bGa ^{\top } \bOm ^{- 1} ) \bBe \non \\
&= ( \bBe - \V )^{\top } \W ( \bBe - \V ) + \Y ^{\top } \U \Y - \V ^{\top } \W \V \non 
\end{align}
for $\W = \X ^{\top } \U \X + \bOm ^{- 1}$ and $\V = ( \X ^{\top } \U \X + \bOm ^{- 1} )^{- 1} ( \X ^{\top } \U \Y + \bOm ^{- 1} \bGa )$. 
Thus, 
\begin{align}
I %
&= C_1 C_2 \int \Big( {1 \over | \bOm |^{d / 2}} {| \bPsi |^{(n + \nu ) / 2} \over | \bSi |^{(n + \nu + d + 1) / 2}} \etr \Big\{ - {1 \over 2} \bSi ^{- 1} ( \bPsi + \bGa ^{\top } \bOm ^{- 1} \bGa ) \Big\} {1 \over | \X ^{\top } \U \X + \bOm ^{- 1} |^{d / 2}} \non \\
&\quad \times \etr \Big\{ - {1 \over 2} \bSi ^{- 1} ( \Y ^{\top } \U \Y - \V ^{\top } \W \V ) \Big\} \non \\
&\quad \times \Big[ \prod_{i = 1}^{n} \{ ( \y _i - \bBe ^{\top } \x _i )^{\top } \bSi ^{- 1} ( \y _i - \bBe ^{\top } \x _i ) / 2 + b \} ^{d / 2 + a} \Big] \Big[ \prod_{i = 1}^{n} \{ {u_i}^{d / 2 + a - 1} \exp (- b u_i ) \} \Big] \non \\
&\quad \times (2 \pi )^{p d / 2} {| \X ^{\top } \U \X + \bOm ^{- 1} |^{d / 2} \over (2 \pi )^{p d / 2} | \bSi |^{p / 2}} \etr \Big\{ - {1 \over 2} \bSi ^{- 1} ( \bBe - \V )^{\top } \W ( \bBe - \V ) \Big\} \Big) d( \bBe , \bSi , \u ) \text{.} \label{pproperp1} 
\end{align}

We have 
\begin{align}
&\prod_{i = 1}^{n} \{ ( \y _i - \bBe ^{\top } \x _i )^{\top } \bSi ^{- 1} ( \y _i - \bBe ^{\top } \x _i ) / 2 + b \} ^{d / 2 + a} \non \\
&\le \Big\{ \prod_{i = 1}^{n} ( {\y _i}^{\top } \bSi ^{- 1} \y _i + {\x _i}^{\top } \bBe \bSi ^{- 1} \bBe ^{\top } \x _i + b) \Big\} ^{d / 2 + a} \non \\
&= \Big[ \prod_{i = 1}^{n} \{ \tr ( \bSi ^{- 1} \y _i {\y _i}^{\top } ) + \tr ( \bBe \bSi ^{- 1} \bBe ^{\top } \x _i {\x _i}^{\top } ) + b \} \Big] ^{d / 2 + a} \non \\
&\le M_3 \{ \tr ( \bSi ^{- 1} ) + \tr ( \bBe \bSi ^{- 1} \bBe ^{\top } ) + 1 \} ^{n (d / 2 + a)} \non \\
&\le M_4 [ \{ \tr ( \bSi ^{- 1} ) \} ^{n (d / 2 + a)} + \{ \tr ( \bBe \bSi ^{- 1} \bBe ^{\top } ) \} ^{n (d / 2 + a)} + 1] \non 
\end{align}
for some $M_3 , M_4 > 0$, where 
\begin{align}
\{ \tr ( \bBe \bSi ^{- 1} \bBe ^{\top } ) \} ^{n (d / 2 + a)} &\le [ 2 \tr \{ ( \bBe - \V ) \bSi ^{- 1} ( \bBe - \V )^{\top } \} + 2 \tr ( \V \bSi ^{- 1} \V ^{\top } )]^{n (d / 2 + a)} \non \\
&\le M_5 [ 2 \tr \{ ( \bBe - \V ) \bSi ^{- 1} ( \bBe - \V )^{\top } \} ]^{n (d / 2 + a)} + M_5 \{ 2 \tr ( \V \bSi ^{- 1} \V ^{\top } ) \} ^{n (d / 2 + a)} \non 
\end{align}
for some $M_5 > 0$. 
Since $\I \le M_6 \A ^{- 1} \non \le M_6 \bOm ^{- 1} \non \le M_6 \W $ for some $M_6 > 0$, 
\begin{align}
\tr ( \V \bSi ^{- 1} \V ^{\top } ) &= \tr ( \bSi ^{- 1} \V ^{\top } \V ) \le M_6 \tr ( \bSi ^{- 1} \V ^{\top } \W \V ) \non \\
&= M_6 \tr \{ \bSi ^{- 1} ( \X ^{\top } \U \Y + \bOm ^{- 1} \bGa )^{\top } ( \X ^{\top } \U \X + \bOm ^{- 1} )^{- 1} ( \X ^{\top } \U \Y + \bOm ^{- 1} \bGa ) \} \text{.} \non 
\end{align}
Therefore, 
\begin{align}
\tr ( \V \bSi ^{- 1} \V ^{\top } ) &\le 2 M_6 \tr [ \bSi ^{- 1} \{ \Y ^{\top } \U \X ( \X ^{\top } \U \X )^{- 1} \X ^{\top } \U \Y + \bGa ^{\top } \bOm ^{- 1} \bOm \bOm ^{- 1} \bGa \} ] \non \\
&\le 2 M_6 \tr [ \bSi ^{- 1} \{ \Y ^{\top } \U \Y + ( \X ^{\top } \U \Y + \A ^{- 1} \B )^{\top } ( \X ^{\top } \U \X + \A ^{- 1} )^{- 1} ( \X ^{\top } \U \Y + \A ^{- 1} \B ) \} ] \non \\
&\le 2 M_6 \tr \{ \bSi ^{- 1} (3 \Y ^{\top } \U \Y + 2 \B ^{\top } \A ^{- 1} \B ) \} \non \\
&\le M_7 [ \{ \tr ( \bSi ^{- 1} ) \} \tr ( \Y ^{\top } \U \Y ) + \tr \bSi ^{- 1} ] \le M_8 \{ \tr ( \bSi ^{- 1} ) \} \{ ( u_1 + \dots + u_n ) + 1 \} \non 
\end{align}
for some $M_7 , M_8 > 0$. 
Thus, 
\begin{align}
&\{ \tr ( \bBe \bSi ^{- 1} \bBe ^{\top } ) \} ^{n (d / 2 + a)} \non \\
&\le M_5 [ 2 \tr \{ ( \bBe - \V ) \bSi ^{- 1} ( \bBe - \V )^{\top } \} ]^{n (d / 2 + a)} + M_9 \{ \tr ( \bSi ^{- 1} ) \} ^{n (d / 2 + a)} \{ ( u_1 + \dots + u_n ) + 1 \} ^{n (d / 2 + a)} \non 
\end{align}
for some $M_9 > 0$. 
Hence, 
\begin{align}
&\prod_{i = 1}^{n} \{ ( \y _i - \bBe ^{\top } \x _i )^{\top } \bSi ^{- 1} ( \y _i - \bBe ^{\top } \x _i ) / 2 + b \} ^{d / 2 + a} \non \\
&\le M_{10} ([ \tr \{ ( \bBe - \V ) \bSi ^{- 1} ( \bBe - \V )^{\top } \} ]^{n (d / 2 + a)} + \{ \tr ( \bSi ^{- 1} ) \} ^{n (d / 2 + a)} \{ 1 + ( u_1 + \dots + u_n )^{n (d / 2 + a)} \} + 1) \non \\
&\le M_{10} (1 + \{ \tr ( \bSi ^{- 1} ) \} ^{n (d / 2 + a)} + [ \tr \{ ( \bBe - \V ) \bSi ^{- 1} ( \bBe - \V )^{\top } \} ]^{n (d / 2 + a)} ) \{ 1 + ( u_1 + \dots + u_n )^{n (d / 2 + a)} \} \label{pproperp2} 
\end{align}
for some $M_{10} > 0$. 

By (\ref{pproperp1}) and (\ref{pproperp2}), 
\begin{align}
I &\le M_{11} \int \Big\{ {1 \over | \bOm |^{d / 2}} {| \bPsi |^{(n + \nu ) / 2} \over | \bSi |^{(n + \nu + d + 1) / 2}} \etr \Big\{ - {1 \over 2} \bSi ^{- 1} ( \bPsi + \bGa ^{\top } \bOm ^{- 1} \bGa ) \Big\} {1 \over | \X ^{\top } \U \X + \bOm ^{- 1} |^{d / 2}} \non \\
&\quad \times \etr \Big\{ - {1 \over 2} \bSi ^{- 1} ( \Y ^{\top } \U \Y - \V ^{\top } \W \V ) \Big\} (1 + \{ \tr ( \bSi ^{- 1} ) \} ^{n (d / 2 + a)} + [ \tr \{ ( \bBe - \V ) \bSi ^{- 1} ( \bBe - \V )^{\top } \} ]^{n (d / 2 + a)} ) \non \\
&\quad \times \Big[ \{ 1 + ( u_1 + \dots + u_n )^{n (d / 2 + a)} \} \prod_{i = 1}^{n} \{ {u_i}^{d / 2 + a - 1} \exp (- b u_i ) \} \Big] \non \\
&\quad \times {| \X ^{\top } \U \X + \bOm ^{- 1} |^{d / 2} \over (2 \pi )^{p d / 2} | \bSi |^{p / 2}} \etr \Big\{ - {1 \over 2} \bSi ^{- 1} ( \bBe - \V )^{\top } \W ( \bBe - \V ) \Big\} \Big\} d( \bBe , \bSi , \u ) \non \\
&\le M_{12} \int \Big( {| \bPsi |^{(n + \nu ) / 2} \over | \bSi |^{(n + \nu + d + 1) / 2}} \etr \Big[ - {1 \over 2} \bSi ^{- 1} \{ \bPsi + \bGa ^{\top } \bOm ^{- 1} \bGa + \Y ^{\top } \U \Y \non \\
&\quad - ( \X ^{\top } \U \Y + \bOm ^{- 1} \bGa )^{\top } ( \X ^{\top } \U \X + \bOm ^{- 1} )^{- 1} ( \X ^{\top } \U \Y + \bOm ^{- 1} \bGa ) \} \Big] \non \\
&\quad \times [1 + \{ \tr ( \bSi ^{- 1} ) \} ^{n (d / 2 + a)} ]^{n (d / 2 + a)} \Big[ \prod_{i = 1}^{n} \{ {u_i}^{d / 2 + a - 1} \exp (- b u_i / 2) \} \Big] \non \\
&\quad \times {| \X ^{\top } \U \X + \bOm ^{- 1} |^{d / 2} \over (2 \pi )^{p d / 2} | \bSi |^{p / 2}} \etr \Big\{ - {1 \over 2} \bSi ^{- 1} ( \bBe - \V )^{\top } ( \W / 2) ( \bBe - \V ) \Big\} \Big) d( \bBe , \bSi , \u ) \non \\
&\le M_{13} \int \Big( {| \C ^{- 1} + \B ^{\top } \A ^{- 1} \B + \Y ^{\top } \U \Y |^{(n + \nu ) / 2} \over | \bSi |^{(n + \nu + d + 1) / 2}} \etr \Big[ - {1 \over 2} \bSi ^{- 1} \{ \bPsi + \bGa ^{\top } \bOm ^{- 1} \bGa + \Y ^{\top } \U \Y \non \\
&\quad - ( \X ^{\top } \U \Y + \bOm ^{- 1} \bGa )^{\top } ( \X ^{\top } \U \X + \bOm ^{- 1} )^{- 1} ( \X ^{\top } \U \Y + \bOm ^{- 1} \bGa ) \} \Big] \non \\
&\quad \times [1 + \{ \tr ( \bSi ^{- 1} ) \} ^{n (d / 2 + a)} ]^{n (d / 2 + a)} \prod_{i = 1}^{n} \{ {u_i}^{d / 2 + a - 1} \exp (- b u_i / 2) \} \Big) d( \bSi , \u ) \non 
\end{align}
for some $M_{11} , M_{12} , M_{13} > 0$. 
Note that 
\begin{align}
&\bPsi + \bGa ^{\top } \bOm ^{- 1} \bGa + \Y ^{\top } \U \Y - ( \X ^{\top } \U \Y + \bOm ^{- 1} \bGa )^{\top } ( \X ^{\top } \U \X + \bOm ^{- 1} )^{- 1} ( \X ^{\top } \U \Y + \bOm ^{- 1} \bGa ) \non \\
&\hspace{- 1cm} = \C ^{- 1} + \B ^{\top } \A ^{- 1} \B + \Y ^{\top } \U \Y + \Y ^{\top } \U \Y - (2 \X ^{\top } \U \Y + \A ^{- 1} \B )^{\top } (2 \X ^{\top } \U \X + \A ^{- 1} )^{- 1} (2 \X ^{\top } \U \Y + \A ^{- 1} \B ) \ge \C \non 
\end{align}
since 
\begin{align}
&\begin{pmatrix} \B ^{\top } \A ^{- 1} \B + 2 \Y ^{\top } \U \Y & (2 \X ^{\top } \U \Y + \A ^{- 1} \B )^{\top } \\ 2 \X ^{\top } \U \Y + \A ^{- 1} \B & 2 \X ^{\top } \U \X + \A ^{- 1} \end{pmatrix} \non \\
&= \begin{pmatrix} \B ^{\top } & \O ^{(d, p)} \\ \O ^{(p)} & \I ^{(p)} \end{pmatrix} \begin{pmatrix} \A ^{- 1} & \A ^{- 1} \\ \A ^{- 1} & \A ^{- 1} \end{pmatrix} \begin{pmatrix} \B & \O ^{(p)} \\ \O ^{(p, d)} & \I ^{(p)} \end{pmatrix} + 2 \begin{pmatrix} \Y ^{\top } & \O ^{(d, n)} \\ \O ^{(p, n)} & \X ^{\top } \end{pmatrix} \begin{pmatrix} \U & \U \\ \U & \U \end{pmatrix} \begin{pmatrix} \Y & \O ^{(n, p)} \\ \O ^{(n, d)} & \X \end{pmatrix} \ge \O ^{(p + d)} \text{.} \non 
\end{align}
Then 
\begin{align}
I &\le M_{14} \int \Big( {|(1 + \tr \U ) \I ^{(d)} |^{(n + \nu ) / 2} \over | \bSi |^{(n + \nu + d + 1) / 2}} \etr \Big( - {1 \over 2} \bSi ^{- 1} \C \Big) [1 + \{ \tr ( \bSi ^{- 1} ) \} ^{n (d / 2 + a)} ]^{n (d / 2 + a)} \non \\
&\quad \times \prod_{i = 1}^{n} \{ {u_i}^{d / 2 + a - 1} \exp (- b u_i / 2) \} \Big) d( \bSi , \u ) \non 
\end{align}
for some $M_{14} > 0$. 
Since the right-hand side of the above inequality is finite, the result follows. 
\hfill$\Box$

\section{Closed-Form Bounds for the Integrals $I$ and $J$}
\label{sec:closed-form-bounds} 
We can in principle obtain closed-form bounds for the integrals $I$ and $J$ of Section \ref{subsec:mixing_time} of the main text, for example, when the initial density $q$ is multivariate normal. 
For $J$, we have 
\begin{align}
( \| \y \| ^2 / n + \bbe ^{\top } \X ^{\top } \X \bbe / n)^{n / 2} q( \bbe ) &\le [ \| \y \| ^2 / n + \{ \tr ( \X ^{\top } \X ) \} \| \bbe \| ^2 / n]^{n / 2} \sup_{\bbe \in \mathbb{R} ^p} q( \bbe ) \non \\
&\le 2^{n / 2} {( \| \y \| ^2 / n )^{n / 2} + [ \{ \tr ( \X ^{\top } \X ) \} \| \bbe \| ^2 / n]^{n / 2} \over 2} \sup_{\bbe \in \mathbb{R} ^p} q( \bbe ) \text{.} \non 
\end{align}
For $I$, we have the following result. 

\begin{lem}
\label{lem:solution} 
For $\i = ( i_1 , \dots , i_{p + 1} )^{\top } \in \{ 1, \dots , n \} ^{p + 1}$ satisfying $1 \le i_1 < \dots < i_{p + 1} \le n$, let $\de _{\i } = \| {\Xbt _{\i }}^{- 1} \ybt _{\i } - {\widetilde{\Xbt } _{\i }}^{- 1} \tilde{\ybt } _{\i } \| / ( \| {\Xbt _{\i }}^{- 1} \| + \| {\widetilde{\Xbt } _{\i }}^{- 1} \| ) > 0$, where $\ybt _{\i } = ( y_{i_1} , \dots , y_{i_p} )^{\top }$ and $\tilde{\ybt } _{\i } = ( y_{i_1} , \dots , y_{i_{p - 1}} , y_{i_{p + 1}} )^{\top }$ and $\Xbt _{\i } = ( \x _{i_1} , \dots , \x _{i_p} )^{\top }$ and $\widetilde{\Xbt } _{\i } = ( \x _{i_1} , \dots , \x _{i_{p - 1}} , \x _{i_{p + 1}} )^{\top }$. 
Let $\de = ( \min_{1 \le i_1 < \dots < i_{p + 1} \le n} \de _{( i_1 , \dots , i_{p + 1} )^{\top }} ) / \sqrt{p + 1} > 0$. 
\begin{itemize}
\item[{\rm{(i)}}]
For all $\bbe \in \mathbb{R} ^p$, there are $1 \le i_1 < \dots < i_{n - p} \le n$ satisfying $\min_{1 \le j \le n - p} | y_{i_j} - {\x _{i_j}}^{\top } \bbe | \ge \de $. 
\item[{\rm{(ii)}}]
We have 
\begin{align}
I &\le \Big( n^p \max_{1 \le i_1 < \dots < i_p \le n} \Big| {1 \over |( \x _{i_1} , \dots , \x _{i_p} )|} \Big| \Big) {B((n - p) a, (n - p) / 2) \{ B(1 / 2, a) \} ^p \over \de^{n - p} b^{n a} / 2^{n / 2}} \text{.} \non 
\end{align}
\end{itemize}
\end{lem}

\noindent
{\bf Proof
.} \ \ For part (i), fix $\bbe \in \mathbb{R} ^p$. 
Then for all $\i = ( i_1 , \dots , i_{p + 1} )^{\top } \in \{ 1, \dots , n \} ^{p + 1}$ satisfying $1 \le i_1 < \dots < i_{p + 1} \le n$, we have 
\begin{align}
\de _{\i } &\le ( \| {\Xbt _{\i }}^{- 1} \ybt _{\i } - \bbe \| + \| {\widetilde{\Xbt } _{\i }}^{- 1} \tilde{\ybt } _{\i } - \bbe \| ) / ( \| {\Xbt _{\i }}^{- 1} \| + \| {\widetilde{\Xbt } _{\i }}^{- 1} \| ) \non \\
&\le ( \| {\Xbt _{\i }}^{- 1} \| \| \ybt _{\i } - \Xbt _{\i } \bbe \| + \| {\widetilde{\Xbt } _{\i }}^{- 1} \| \| \tilde{\ybt } _{\i } - \widetilde{\Xbt } _{\i } \bbe \| ) / ( \| {\Xbt _{\i }}^{- 1} \| + \| {\widetilde{\Xbt } _{\i }}^{- 1} \| ) \non \\
&\le \| ( y_{i_1} , \dots , y_{i_{p + 1}} )^{\top } - ( \x _{i_1} , \dots , \x _{i_{p + 1}} )^{\top } \bbe \| \text{,} \non 
\end{align}
which implies that $| y_{i_j} - {\x _{i_j}}^{\top } \bbe | \ge \de _{i_1 , \dots , i_{p + 1}} / \sqrt{p + 1} \ge \de $ for some $j = 1, \dots , p + 1$. 
Therefore, there exist $1 \le i_1 < \dots < i_{n - p} \le n$ such that for all $j = 1, \dots , n - p$, we have $| y_{i_j} - {\x _{i_j}}^{\top } \bbe | \ge \de $. 

For part (ii), we have, by part (i), that 
\begin{align}
&\prod_{i = 1}^{n} {1 \over \{ ( y_i - {\x _i}^{\top } \bbe )^2 / 2 + b \sis \} ^{1 / 2 + a}} \non \\
&\le {1 \over ( \de ^2 / 2 + b \sis )^{(n - p) (1 / 2 + a)}} \sum_{1 \le i_1 < \dots < i_p \le n} \prod_{k = 1}^{p} {1 \over \{ ( y_{i_k} - {\x _{i_k}}^{\top } \bbe )^2 / 2 + b \sis \} ^{1 / 2 + a}} \non 
\end{align}
for all $( \bbe , \sis ) \in \mathbb{R} ^p \times (0, \infty )$. 
Therefore, 
\begin{align}
I &\le \sum_{1 \le i_1 < \dots < i_p \le n} \int_{\mathbb{R} ^p \times (0, \infty )} {( \sis )^{n a - 1} \over ( \de ^2 / 2 + b \sis )^{(n - p) (1 / 2 + a)}} \Big[ \prod_{k = 1}^{p} {1 \over \{ ( y_{i_k} - {\x _{i_k}}^{\top } \bbe )^2 / 2 + b \sis \} ^{1 / 2 + a}} \Big] d( \bbe , \sis ) \non \\
&= \sum_{1 \le i_1 < \dots < i_p \le n} \int_{\mathbb{R} ^p \times (0, \infty )} {( \sis )^{n a - 1} / ( \de ^2 / 2 + b \sis )^{(n - p) (1 / 2 + a)} \over \prod_{k = 1}^{p} [ \{ ( \e _{k}^{(p)} )^{\top } \bbet \} ^2 / 2 + b \sis ]^{1 / 2 + a}} \Big| {1 \over |( \x _{i_1} , \dots , \x _{i_p} )|} \Big| d( \bbet , \sis ) \non \\
&= \sum_{1 \le i_1 < \dots < i_p \le n} \int_{\mathbb{R} ^p \times (0, \infty )} {( \sis )^{(n - p) a - 1} / ( \de ^2 / 2 + b \sis )^{(n - p) (1 / 2 + a)} \over \prod_{k = 1}^{p} [ \{ ( \e _{k}^{(p)} )^{\top } \widetilde{\bbet } \} ^2 / 2 + b]^{1 / 2 + a}} \Big| {1 \over |( \x _{i_1} , \dots , \x _{i_p} )|} \Big| d( \widetilde{\bbet } , \sis ) \non \\
&\le \Big( n^p \max_{1 \le i_1 < \dots < i_p \le n} \Big| {1 \over |( \x _{i_1} , \dots , \x _{i_p} )|} \Big| \Big) \int_{0}^{\infty } {( \sis )^{(n - p) a - 1} \over ( \de ^2 / 2 + b \sis )^{(n - p) (1 / 2 + a)}} d\sis \Big\{ \int_{- \infty }^{\infty } {1 \over ( \be ^2 / 2 + b)^{1 / 2 + a}} d\be \Big\} ^p \text{,} \non %
\end{align}
from which the desired result follows. 
\hfill$\Box$

\begin{remark}
The constant 
\begin{align}
&{1 \over \de^{n - p}} \max_{1 \le i_1 < \dots < i_p \le n} \Big| {1 \over |( \x _{i_1} , \dots , \x _{i_p} )|} \Big| \non 
\end{align}
can be very large in general. 
Therefore, a better bound would be needed in practice. 
\end{remark}

\end{document}